\documentclass[a4paper,11pt]{article}

\usepackage[headinclude]{typearea}
\usepackage[english]{babel}           
\usepackage[utf8]{inputenc}
\usepackage[T1]{fontenc}
\usepackage{scrtime}
\usepackage{amsmath,amsthm,amsfonts,amssymb}
\usepackage {color}                                  
\usepackage {pifont}                                 
\usepackage {multicol}                               
\usepackage {times}                                  

\usepackage[numbers]{natbib}                                                     
\usepackage[normalem]{ulem}                                                     
\usepackage [scaled=.90]{helvet}                     
\usepackage {courier}                                
\usepackage {graphicx}                               
\usepackage {array}                                  
\usepackage {longtable}                              
\usepackage {makeidx}                                
\usepackage{fancyvrb}
\usepackage{enumerate} 
\usepackage{ifpdf}
\usepackage{caption}
\usepackage{mathrsfs}

\usepackage{setspace}

\usepackage{marvosym}

\theoremstyle{plain}
\newtheorem{theorem}{Theorem}[section]

\newtheorem{lemma}[theorem]{Lemma}
\newtheorem{newlemma}{Lemma}[section]

\newtheorem{proposition}[theorem]{Proposition}
\newtheorem{conjecture}[theorem]{Conjecture}

\theoremstyle{definition}
\newtheorem{assumption}{Assumption}[section]
\newtheorem{definition}[theorem]{Definition}
\newtheorem{example}[theorem]{Example}
\newtheorem{remark}[theorem]{Remark}
\newtheorem{algorithm}[theorem]{Algorithm}

\newcommand{\exclude}[1]{}
\newcommand{\opA}{\cA}
\newcommand{\opB}{\cB}
\newcommand{\param}{\mathscr T}
\newcommand{\paraminv}{\mathscr S}
\newcommand{\proj}{p}
\newcommand{\tang}{\tau}
\newcommand{\mm}{\cI_\xi}
\newcommand{\vv}{b}
\newcommand{\thp}{\vartheta}
\newcommand{\err}{{\operatorname{err}}}
\newcommand{\id}{{\operatorname{id}_{\R^d}}}
\newcommand{\idz}{{\operatorname{id}_{\R^2}}}
\newcommand{\idV}{{\operatorname{id}_{V}}}
\newcommand{\idt}{{\operatorname{id}_{\tang(\xi)}}}
\newcommand{\reach}{{\operatorname{reach}}}
\newcommand{\1}{{\mathbf{1}}}
\newcommand{\E}{{\mathbb{E}}}

\newcommand{\N}{{\mathbb{N}}}
\renewcommand{\P}{{\mathbb{P}}}

\newcommand{\R}{{\mathbb{R}}}

\newcommand{\sign}{\operatorname{sign}}

\newcommand{\subtheta}{B}

\definecolor{darkgreen}{rgb}{0,0.5,0}
\definecolor{middlegreen}{rgb}{0,0.7,0}

\definecolor{magenta}{rgb}{0.75,0,0.25}

\definecolor{violet}{rgb}{0.25,0,0.75}

\newcommand{\hypsurf}{\Theta}

\newcommand{\tr}{\operatorname{tr}}

\renewcommand{\P}{{\mathbb P}}

\newcommand{\cA}{{\cal A}}
\newcommand{\cB}{{\cal B}}

\newcommand{\cF}{{\cal F}}

\newcommand{\cI}{{\cal I}}

\newcommand{\be}{\begin{equation}}
\newcommand{\ee}{\end{equation}}
\newcommand{\bea}{\begin{eqnarray}}
\newcommand{\eea}{\end{eqnarray}}
\newcommand{\beast}{\begin{eqnarray*}}
\newcommand{\eeast}{\end{eqnarray*}}
\newcommand{\bproof}{\begin{proof}}
\newcommand{\eproof}{\end{proof}}

\title{A Strong Order 1/2 Method for Multidimensional SDEs with 
Discontinuous Drift}

\author{Gunther Leobacher and Michaela Sz\"olgyenyi}

\begin{document}

\date{Corrected version, December 2018}

\maketitle


\begin{abstract}
In this paper we consider multidimensional
stochastic differential equations (SDEs) with discontinuous drift and possibly
degenerate diffusion coefficient. 
We prove an existence and uniqueness result for this class of SDEs
and
we present a numerical method that converges with strong order 1/2.
Our result is the first one that shows existence and uniqueness as well as strong convergence for such a general
class of SDEs.

The proof is based on a transformation technique that removes the discontinuity
from the drift such that the coefficients of the transformed SDE are Lipschitz
continuous.  Thus the Euler-Maruyama method can be applied to this transformed 
SDE. The approximation can be transformed back, giving an approximation to
the solution of the original SDE.

As an illustration, we apply our result to an SDE the drift of which has a discontinuity along the unit circle and we present an application from stochastic optimal control.\\
 
\end{abstract}


\section{Introduction}
\label{sec:Introduction}

We consider a $d$-dimensional time-homogeneous stochastic differential equation (SDE)
\begin{align}\label{eq:SDE}
dX &= \mu(X) \, dt + \sigma (X) dW\,, \qquad X_0=x\,,
 \end{align}
where $\mu:\R^d\longrightarrow \R^d$ and $\sigma:\R^d \longrightarrow
\R^{d\times d}$ are measurable functions and $W=(W_t)_{t\ge0}$ is a
$d$-dimensional standard Brownian motion on the filtered probability space
$(\Omega,\cF,(\cF_t)_{t\ge0},\P)$.  \\
 
If both $\mu$ and $\sigma$ are Lipschitz, then existence and uniqueness is
guaranteed by Picard iteration.
Furthermore, \eqref{eq:SDE} can be solved numerically with, e.g., the Euler-Maruyama method, which then converges
with strong order
1/2, see \cite[Theorem 10.2.2]{kloeden1992}.\\

However, in applications one is frequently confronted with SDEs where $\mu$ is non-Lipschitz, e.g., in stochastic control theory. There, whenever an optimal control of bang-bang type appears, meaning that the strategy is of the form 
$\1_S(X)$ for some measurable set $S\subseteq \R^d$, the drift of the controlled underlying system is discontinuous. Furthermore, for example in setups with incomplete information, which are currently heavily under study, e.g., for applications in mathematical finance, the underlying systems have degenerate diffusion coefficients.
Therefore, the class of SDEs that we study in this paper appears frequently in applied mathematics and we shall elaborate our contributions to this kind of problems later in the paper.\\

The question of existence and uniqueness of solutions to SDEs with non-Lipschitz drift has been studied by various authors.

For the case where $\mu$ is only bounded and measurable
and $\sigma$ is bounded, Lipschitz, and 
satisfies a certain uniform ellipticity condition, \citet{zvonkin1974} and
\citet{veretennikov1981,veretennikov1982,veretennikov1984} 
prove existence and uniqueness of a solution by  removing the drift coefficient in a way such that the Lipschitz
condition of the diffusion coefficient is preserved.

But uniform ellipticity is a strong assumption which is -- as mentioned above -- frequently violated in applications.

In \citet{sz14} an existence and uniqueness result for  
\eqref{eq:SDE} is presented for the case where the drift is potentially
discontinuous at a hyperplane, or a special hypersurface, but well behaved
everywhere else and where the diffusion coefficient is
potentially degenerate. In that paper, not the whole drift is removed, but
only the discontinuity is removed locally from the drift.

Due to the weaker requirements on the diffusion coefficient the restriction to homogeneous SDEs does not pose any loss of generality. In \citet{sz2016a} the authors extend the result from \cite{sz14} to the time-inhomogeneous case.

In \citet{sz15} an existence and uniqueness result, as well as a numerical
method are presented for the one-dimensional
case with piecewise Lipschitz drift coefficient.  There the
coefficients are globally transformed into Lipschitz ones. Both
computation of the transformed coefficients and inversion can be done
efficiently.  This leads to a numerical method for one-dimensional SDEs 
through application of the Euler-Maruyama
scheme on the transformed equation and transforming the approximation back.
We present a simplified version of this result in Section \ref{sec:1-dimensional}.
\\

However, extending the result from \cite{sz15} to the $d$-dimensional case is
 far from being straightforward. 
One problem is that there is no immediate generalization of the concept
of a piecewise Lipschitz function with several variables that suits our needs. 
The second problem is that it is more difficult to obtain a transform that 
is a Lipschitz diffeomorphism $\R^d\longrightarrow\R^d$. We use Hadamard's
global inverse function theorem to prove that our transform is of this kind.
Moreover, we need to show that the transform and its inverse are sufficiently
well-behaved for It\^o's formula to hold. 

The coefficients of the SDE obtained by transforming the original one
are shown to be
Lipschitz, such that we can apply the
Euler-Maruyama method to the transformed SDE. An approximation to the original SDE is then obtained
by applying the inverse transform to the approximation of the transformed solution.
 For this scheme we show 
strong convergence with order $1/2$.
One might ask whether the results of Zvonkin and
Veretennikov give rise to a similar method. However, in order to apply their method one
would have to solve a system of parabolic partial differential equations (in
each step). Further, for using this solution in a numerical method
like ours, one would also have to find its inverse function. Therefore such a method, if it exists at all,
would be rather costly from the computational perspective.\\
 
In the present paper we present a transform for the multidimensional case which
allows to prove an existence and uniqueness result for
 $d$-dimensional SDEs with discontinuous drift and degenerate diffusion coefficient under conditions significantly weaker than those in 
the literature. The essential geometric condition in our setup
is that the diffusion must have a component orthogonal
to the set of discontinuities of the drift.

Furthermore, we 
present a numerical method for such SDEs based on the ideas outlined above.
Up to the authors' knowledge there is no other numerical method that can deal
with such a general class of SDEs and gives strong convergence, much less
giving  a strong convergence {\em rate}.\\

We are now going to review the literature on numerical methods for SDEs with non-globally Lipschitz drift coefficient.
In \citet{berkaoui2004} strong convergence of the Euler-Maruyama scheme is proven under the assumption that the drift is of class $C^1$. 
For an SDE with continuously differentiable but non-globally Lipschitz drift \citet{hutzenthaler2012} introduce a new explicit numerical scheme -- the tamed Euler scheme -- and prove its strong convergence.
\citet{sabanis2013} proves strong convergence of the tamed Euler scheme for SDEs with one-sided Lipschitz drift.
For the Euler-Maruyama scheme \citet{gyongy1998} proves almost sure convergence for the case that the drift satisfies a monotonicity condition.
A different approach is introduced by \citet{halidias2008}, who show that the Euler-Maruyama scheme converges
strongly for SDEs with a discontinuous monotone drift coefficient, especially mentioning the case in which the drift is a Heaviside function.  
\citet{kohatsu2013} show weak convergence of a method where they first regularize the drift and then apply the Euler-Maruyama scheme. They allow the drift to be discontinuous.
\citet{etore2013,etore2014} introduce an exact simulation algorithm for
one-dimensional SDEs that have a bounded drift coefficient being discontinuous
in one point, but differentiable everywhere else.

This paper is organized as follows. 
In Section \ref{sec:1-dimensional} we present the one-dimensional result and
algorithm in a form that can be generalized to multiple dimensions,
which is subsequently done in Section \ref{sec:d-dimensional}.
In Section \ref{sec:Example} we give two numerical examples: one where the drift coefficient has discontinuities along the unit circle in $\R^2$ and an example from stochastic optimal control.

Some of the more technical and geometrical proofs have been moved to the
appendix.


\section{The one-dimensional problem}\label{sec:1-dimensional}

Here we consider the one-dimensional version of SDE \eqref{eq:SDE} 
and give simple conditions for existence and uniqueness of a solution and a 
strong order 1/2 algorithm.
For this we recall the following definition.

\begin{definition}\label{def:pwlip-1dim}
Let $I\subseteq\R$ be an interval. We say a function $f:I\longrightarrow\R$
is piecewise Lipschitz, if there are finitely many points $\xi_1<\ldots<\xi_m\in
I$ such that $f$ is Lipschitz on each of the intervals $(-\infty,\xi_1)\cap I,
(\xi_m,\infty)\cap I$ and $(\xi_k,\xi_{k+1}),\,k=1,\ldots,m$.
\end{definition}

We make the following assumptions on the coefficients.

\begin{assumption}\label{ass:mu-1dim}
The drift coefficient $\mu:\R \longrightarrow \R$ is piecewise Lipschitz.
\end{assumption} 
\begin{assumption}\label{ass:sigma-1dim}
The diffusion coefficient $\sigma:\R\longrightarrow\R$ is Lipschitz with $\sigma(\xi) \neq 0$ whenever $\mu(\xi+) \neq \mu(\xi-)$.
\end{assumption}

For simplicity we derive the result for $\mu:\R\longrightarrow\R$ that is
Lipschitz with the exception of only a single point $\xi$ where $\mu$ is
allowed to
jump.
We are going to construct a transform $G:\R\longrightarrow \R$ such that 
the process formally defined by $Z=G(X)$ satisfies an SDE with 
Lipschitz coefficients and therefore has a solution by It\^o's classical
theorem on existence and uniqueness of solutions, see \cite{ito1951}.\\

For this define the following bump function on $\R$, which we need to localize the impact of the transform $G$:  
\begin{align}\label{eq:bump}
\phi(u)=
\begin{cases}
(1+u)^3(1-u)^3 & \text{if } |u|\le 1\,,\\
0 & \text{else}\,.
\end{cases}
\end{align}
The function $\phi$ has the following properties:
\vspace{-.5em}
\begin{enumerate}\setlength\itemsep{0em}
\item\label{it:phi1} $\phi$ defines a $C^2$ function on all of $\R$;
\item\label{it:phi2} $\phi(0)=1$, $\phi'(0)=0$, $\phi''(0)=-6$; 
\item\label{it:phi3} $\phi(u)=\phi'(u)=\phi''(u)=0$ for all $|u|\ge 1$.
\label{pg:remark123}
\end{enumerate}

We define the transform $G:\R\longrightarrow \R$ by
\begin{align}\label{eq:G-1d}
G(x)=x+\alpha \phi\left(\frac{x-\xi}{c}\right)(x-\xi)|x-\xi| \,,\quad x\in \R\,,
\end{align}
where $\alpha\ne 0$ and $c>0$ are some constants.

\begin{lemma}\label{lem:c-1dim}
 Let
$c<\frac{1}{6|\alpha|}$.

Then $G'(x)>0$ for all $x\in \R$. Furthermore, $G'(x)=1$ for all $|x-\xi|>c$.  Therefore $G$ has a global inverse $G^{-1}$.
\end{lemma}

\begin{proof}
Differentiating $G$ for $|x-\xi|\le c$ yields
\begin{align*}
 G'(x)=1-\frac{6 \alpha}{c^2}(x-\xi)^2 |x-\xi| \left(1+\frac{x-\xi}{c}\right)^2 \left(1-\frac{x-\xi}{c}\right)^2+2 \alpha |x-\xi| \left(1+\frac{x-\xi}{c}\right)^3 \left(1-\frac{x-\xi}{c}\right)^3\,.
\end{align*}
For positive $\alpha$ this is positive, if $c<\frac{1}{6|\alpha|}$.
For negative $\alpha$ it is positive, if $c<\frac{1}{2|\alpha|}$.
Altogether a sufficient condition for $G'$ to be positive is $c<\frac{1}{6|\alpha|}$.
\end{proof}

 W.l.o.g. we always choose 
$c<\frac{1}{6|\alpha|}$, such that $G$ has a global inverse.

\begin{remark}
In \cite{sz15} the function $G$ is constructed differently. There, $G$ is
piecewise cubic, such that $G^{-1}$ is piecewise radical and hence admits
exact inversion, which is advantageous for the numerical treatment.

In fact, $G$ can be made piecewise cubic by still using equation \eqref{eq:G-1d},
but with a different choice for $\phi$.
Actually, any function $\phi$ with support contained in $[-1,1]$ satisfying
properties \ref{it:phi1}., \ref{it:phi2}., \ref{it:phi3}. from page \pageref{pg:remark123} 
will give rise to a transform $G$ sufficient for our purpose, with 
a similar condition on the constant $c$ for $G$ to be invertible.
The form chosen here is simple in the one-dimensional
case and
has a direct multidimensional analog.
\end{remark}

Formally define $Z=G(X)$.
Abbreviating $\bar\phi(x):=\phi(\frac{x-\xi}{c})(x-\xi)|x-\xi| $, we have
\begin{align}
dZ&=dX+\alpha \bar\phi'(X)dX+\frac{1}{2}\alpha\bar\phi''(X)d[X]\nonumber\\
&=\left(\mu(X)+\alpha\bar \phi'(X) \mu(X)+\frac{1}{2}\alpha\bar\phi''(X)\sigma(X)^2\right) dt+\left(\sigma(X) +\alpha\bar \phi'(X)\sigma(X)\right)dW\nonumber\\
&=\tilde \mu(Z) dt +\tilde\sigma(Z) dW\,, \label{eq:Z-1d}
\end{align}
where
\begin{align*}
\tilde\mu(z)&=\mu(G^{-1}(z))+\alpha\bar \phi'(G^{-1}(z)) \mu(G^{-1}(z))+\frac{1}{2}\alpha\bar\phi''(G^{-1}(z))\sigma(G^{-1}(z))^2\,,\\
\tilde\sigma(z)&=\sigma(G^{-1}(z)) +\alpha\bar \phi'(G^{-1}(z))\sigma(G^{-1}(z))\,.
\end{align*}

We now show
that, for an appropriate choice of $\alpha$, the transformed drift 
$\tilde \mu$ is Lipschitz.
For this we need the following elementary lemma from \cite{sz15}.

\begin{lemma}\label{th:lipschitz}
Let $f:\R\longrightarrow\R$ be piecewise Lipschitz and continuous.

Then $f$ is Lipschitz on $\R$. 
\end{lemma}

\exclude{
\begin{proof}
The elementary proof can be found in \cite{sz15}.
\end{proof}
}

From Lemma \ref{th:lipschitz} and $\lim_{h\to 0}\bar \phi'(h)=0$ we see that the mapping 
$z\mapsto \bar \phi'(G^{-1}(z)) \mu(G^{-1}(z))$ is Lipschitz. In order to make the mapping
$z\mapsto \mu(G^{-1}(z))+\frac{1}{2}\alpha\bar\phi''(G^{-1}(z))\sigma(G^{-1}(z))^2$ continuous, we
need to choose $\alpha$ so that 
\[
\mu(G^{-1}(\xi+))+\frac{1}{2}\alpha\bar\phi''(G^{-1}(\xi+))\sigma(G^{-1}(\xi+))^2
=\mu(G^{-1}(\xi-))+\frac{1}{2}\alpha\bar\phi''(G^{-1}(\xi-))\sigma(G^{-1}(\xi-))^2\,,
\]
i.e.
\[
\mu(\xi+)+\frac{1}{2}\alpha\bar\phi''(\xi+)\sigma(\xi)^2
=\mu(\xi-)+\frac{1}{2}\alpha\bar\phi''(\xi-)\sigma(\xi)^2\,.
\]
Thus we get, for the choice
\[
\alpha=-2\frac{\mu(\xi+)-\mu(\xi-)}{(\bar\phi''(\xi+)-\bar\phi''(\xi-))\sigma(\xi)^2}
=\frac{\mu(\xi-)-\mu(\xi+)}{2\sigma(\xi)^2}
\]
that $\tilde \mu$ is continuous.
Note that at this point we need non-degeneracy of $\sigma$ in $\xi$.\\

Since $\tilde \mu$ is continuous with the appropriate choice of $\alpha$,
it is Lipschitz as well by Lemma \ref{th:lipschitz}.

One may worry about the quadratic occurrence of $\sigma$ in the expression
for $\tilde \mu$. Note, however, that $\bar\phi''$ vanishes outside $[-c,c]$.\\

To prove that $\tilde \sigma$ is Lipschitz as well, we need the following lemma:

\begin{lemma}\label{th:lip1}
Let $f:\R\longrightarrow\R$ be Lipschitz. 
Then 
$f\phi'$ is Lipschitz.
\end{lemma}

\begin{proof}
Let $L_f$ be a Lipschitz constant for $f$. Note that $6$ is a Lipschitz constant
for $\phi'$. If $|x|,|y|\le 1$, then 
\begin{align*}
|f(x)\phi'(x)-f(y)\phi'(y)|
&\le|f(x)\phi'(x)-f(y)\phi'(x)|+|f(y)\phi'(x)-f(y)\phi'(y)|\\
&\le L_f |x-y|\, \max_{z\in [-1,1]}|\phi'(z)|+6|x-y| \max_{z\in [-1,1]}|f(z)|\\
&\le K |x-y|\,,
\end{align*}
where $K=L_f |x-y|\, \max_{z\in [-1,1]}|\phi'(z)|+6|x-y| \max_{z\in [-1,1]}|f(z)|$.
For $-1\le x\le 1< y$ we have 
\begin{align*}
|f(x)\phi'(x)-f(y)\phi'(y)|
&=|f(x)\phi'(x)|=|f(x)\phi'(x)-f(1)\phi'(1)|
\le K |x-1| \le K |x-y|\,.
\end{align*}
 The same  estimate holds for the case $x<-1\le y  \le 1$.
For $|x|,|y|>1$ we have $|f(x)\phi'(x)-f(y)\phi'(y)|=0\le K |x-y|$.
\end{proof}

Thus, $\tilde\sigma$ is Lipschitz by Lemma \ref{th:lip1} and the fact that the
composition of Lipschitz functions is Lipschitz. 
\\

Altogether we have that the SDE \eqref{eq:Z-1d} for $Z$ has Lipschitz coefficients $\tilde\mu$ and $\tilde \sigma$.\\

The generalization to finitely many discontinuities of $\mu$ in the points 
$\xi_1<\dots<\xi_m$ is now 
straightforward: define 
\[
G(x):=x+\sum_{k=1}^m \alpha_k \phi\Big(\frac{x-\xi_k}{c}\Big)(x-\xi_k)|x-\xi_k|\,,
\]
with
\[
\alpha_k
=\frac{\mu(\xi_k-)-\mu(\xi_k+)}{2\sigma(\xi_k)^2}
\qquad\text{and}\qquad 
c<\min\Big(\min\limits_{1\le k\le m}\frac{1}{6|\alpha_k|},\min\limits_{1\le k\le m-1}\frac{\xi_{k+1}-\xi_k}{2}\Big)\,.
\]

We are ready to prove existence and uniqueness of a solution to the one-dimensional SDE \eqref{eq:SDE}.

\begin{theorem}[\mbox{cf.~\cite[Theorem 2.2]{sz16}}]
\label{exun-1d}
Let Assumptions \ref{ass:mu-1dim}, and \ref{ass:sigma-1dim}
be satisfied, i.e.~$\mu$ is piecewise Lipschitz with finitely many jump points, $\sigma$ is Lipschitz and $\forall \xi: \mu(\xi+)\neq \mu(\xi-) \Longrightarrow \sigma(\xi)\neq 0$.  

Then the
one-dimensional SDE \eqref{eq:SDE} has a unique global strong solution.
\end{theorem}

\begin{proof}
 Since the SDE \eqref{eq:Z-1d} for $Z$ has Lipschitz coefficients, it follows
that \eqref{eq:Z-1d} with initial condition $Z_0=G(x)$  has a unique global
strong solution.  Furthermore, $G$ has a global inverse $G^{-1}$, which
inherits the smoothness from $G$.
 Although $G^{-1} \notin C^2$, It\^o's formula holds for $G^{-1}$, see \cite[5. Problem 7.3]{karatzas1991}.
 Applying It\^o's formula to $G^{-1}$, we obtain that $G^{-1}(Z)$ satisfies
\begin{align*}
dX=\mu(X)dt+\sigma(X)dW\,,\quad X_0=x\,.
\end{align*}
Setting $X=G^{-1}(Z)$ yields the desired result.
\end{proof}

For approximating the solution to the one-dimensional SDE \eqref{eq:SDE} we propose the following numerical method.
Let $Z_T^{(\delta)}$ be the Euler-Maruyama approximation of the solution to SDE \eqref{eq:Z-1d} with step size smaller than $\delta>0$.
\begin{algorithm}\label{alg:num-1d}
Go through the following steps:
\begin{enumerate}
 \item Set $Z^{(\delta)}_0=G(x)$.
 \item Apply the Euler-Maruyama method to the SDE \eqref{eq:Z-1d} to obtain $Z_T^{(\delta)}$.
 \item Set $\bar X = G^{-1}\left(Z_T^{(\delta)}\right)$.
\end{enumerate}
\end{algorithm}

\begin{theorem}[\mbox{cf.\cite[Theorem 3.1]{sz16}}]\label{thm:conv-1d}
Let Assumptions \ref{ass:mu-1dim}, and \ref{ass:sigma-1dim} be satisfied.

Then Algorithm \ref{alg:num-1d}
converges with strong order $1/2$ to the solution X of the one-dimensional SDE \eqref{eq:SDE}.
\end{theorem}

\begin{proof}
We estimate the $L^2$-error of the approximation.
For every $T>0$ there is a constant $C$, such that 
 \begin{align*}
\E\left(\left(X_T- \bar X_T\right)^2\right)
= \E\left(\left(G^{-1}\left(Z_T\right)- G^{-1}\left(Z_T^{(\delta)}\right)\right)^2\right)
\le L_{G^{-1}}^2\E\left(\left(Z_T- Z_T^{(\delta)}\right)^2\right)= L_{G^{-1}}^2C \delta
\end{align*}
for every sufficiently small step size $\delta$, where $L_{G^{-1}}$ is the Lipschitz constant of $G^{-1}$ and where we applied \cite[Theorem 10.2.2]{kloeden1992} for the $L^2$-convergence of
the Euler-Maruyama scheme for SDEs with Lipschitz coefficients.
\end{proof}


\section{The multidimensional problem}\label{sec:d-dimensional}

We now consider the multidimensional case.
Like in dimension one, we will have to make assumptions on the drift so
that it is Lipschitz apart from  -- relatively few -- locations of 
discontinuity.
That is, we need a concept similar to that of ``piecewise Lipschitz''
in the one-dimensional case. We will develop such a concept now.

In contrast to the one-dimensional case, we shall have to make additional 
assumptions on the behaviour of the drift close to its points of
discontinuity, which shall all lie in a hypersurface $\hypsurf$.

Regarding the diffusion coefficient we need to find a condition corresponding
to Assumption \ref{ass:sigma-1dim}.

Note that most of these assumptions 
are automatically satisfied, or can at least be weakened, if $\hypsurf$ is compact.
We will treat the case of compact $\hypsurf$ in Section \ref{subsec:compact}.

\subsection{Piecewise Lipschitz functions}

For a continuous curve $\gamma:[0,1]\longrightarrow \R^d$, let  $\ell(\gamma)$ denote its length,
\[
\ell(\gamma):=\sup_{n,0\le t_1<\ldots<t_n\le 1}\sum_{k=1}^n \|\gamma(t_k)-\gamma(t_{k-1})\|\;\in[0,\infty]\,.
\] 

\begin{definition}
Let $A\subseteq \R^d$. The {\em intrinsic metric} $d$ on $A$ 
is given by 
\[
\rho(x,y):=\inf\{\ell(\gamma):\gamma:[0,1]\rightarrow A \text{ is a  continuous curve satisfying } \gamma(0)=x,\, \gamma(1)=y\}\,,
\]
where $\rho(x,y):=\infty$, if there is no continuous curve from $x$
to $y$.  
\end{definition}

\begin{definition}\label{def:intLip}
Let $A\subseteq\R^d$. Let $f:A\longrightarrow \R^m$ be a function.
We say that $f$ is {\em intrinsic Lipschitz}, if it is Lipschitz w.r.t. the
intrinsic metric on $A$, i.e. if there exists a constant $L$ such that
\[
\forall x,y\in A: \|f(x)-f(y)\|\le L \rho(x,y)\,.
\]
\end{definition}

\begin{remark}
Note that for a function $f:\R\longrightarrow\R$ we have that $f$ is piecewise
Lipschitz, iff $f$ is intrinsic Lipschitz on $\R\backslash B$, where $B$ is a
finite
subset of $\R$. 
\end{remark}

This motivates the following definition:

\begin{definition}\label{def:pw-lip}
A function $f:\R^d\longrightarrow\R^m$ is {\em piecewise Lipschitz}, if
there exists a hypersurface\footnote{By a hypersurface we mean a 
$(d-1)$-dimensional 
submanifold of the $\R^d$.} $\hypsurf$ with finitely many components and 
with the property, that
the restriction $f|_{\R^d\backslash \hypsurf}$ is intrinsic Lipschitz.
We call $\hypsurf$ an {\em exceptional set} for $f$.
\end{definition}

The definition is more general than the more obvious requirement
that $\R^d$  can be partitioned into finitely many patches in a way
such that $f$ is Lipschitz on all of the patches. This is illustrated by
the following example. 

\begin{example}\label{ex:intLip}
Consider the function
$f:\R^2\longrightarrow\R$, $f(x)=\|x\|\arg(x)$. Then $f$ is not Lipschitz, since 
$\lim_{h\to 0+}f(\cos(\pi-h),\sin(\pi-h))=\pi$ and 
$\lim_{h\to 0+}f(\cos(\pi+h),\sin(\pi+h))=-\pi$ for $x_1<0$.

It is readily checked, however, that $f$ is intrinsic Lipschitz on $A=\R^2\backslash\{x\in\R^2: x_1<0,x_2=0\}$ and $\{x\in\R^2: x_1<0,x_2=0\}$ is obviously a 
one-dimensional submanifold of $\R^2$.

Thus $f$ is piecewise Lipschitz in the sense of Definition \ref{def:pw-lip}.
\end{example}

The following lemma is a multidimensional
generalization of  Lemma \ref{th:lipschitz}.

\begin{lemma} \label{th:intrinsic_lip}
Let $f:\R^d\longrightarrow \R^m$ be a function. If
\begin{enumerate}
\item $f$ is continuous in every point $x\in \R^d$;
\item $f$ is piecewise Lipschitz with exceptional set $\hypsurf$;
\item \label{it:conj} for $x,y\in \R^d$ and $\eta>0$ 
there exists a continuous curve $\gamma$ from $x$ to $y$ with $\ell(\gamma)<\|x-y\|+\eta$ such that $\#(\gamma\cap \hypsurf)<\infty$.
\end{enumerate}
Then 
$f$ is Lipschitz on $\R^d$ w.r.t. the Euclidean metric, and with the same Lipschitz constant.
\end{lemma}

\begin{proof}
Let $L$ be the intrinsic Lipschitz constant of $f$, i.e.~$\|f(y)-f(x)\|\le L\rho(x,y)$ for all $x,y\in \R^d$, and let $x,y\in \R^d$. 
If $\rho(x,y)=\|x-y\|$, then clearly $\|f(x)-f(y)\|\le L \rho(x,y) =L\|x-y\|$.

If $\rho(x,y)>\|x-y\|$, then the line segment $s(x,y):=\{(1-\lambda)x+\lambda y:\lambda\in [0,1]\}$ has non-empty intersection with $\hypsurf$.

Consider first the case where $s(x,y)\cap \hypsurf=\{z_1,\ldots,z_n\}$, i.e.~we have finite
intersection. There exist $\lambda_1,\ldots,\lambda_n$ such that
$z_k=(1-\lambda_k)x+\lambda_k y$.
Define $g:[0,1]\longrightarrow \R^m$ by $g(\lambda):=f((1-\lambda)x+\lambda y)$.

Set $z_0=x,z_{n+1}=y, \lambda_0=0,\lambda_{n+1}=1$. W.l.o.g.,
$\lambda_0<\ldots<\lambda_{n+1}$. Now
\begin{align*}
\|f(y)-f(x)\|&=\|\sum_{k=1}^{n+1} f(z_k)-f(z_{k-1})\|
\le \sum_{k=1}^{n+1} \|f(z_k)-f(z_{k-1})\|
=\sum_{k=1}^{n+1} \|g(\lambda_k)-g(\lambda_{k-1})\|\\
&=\lim_{h\to 0+}\sum_{k=1}^{n+1} \|g(\lambda_k-h)-g(\lambda_{k-1}+h)\|\\
&\le\lim_{h\to 0+}\sum_{k=1}^{n+1} L 
\rho\Big(\big((1-\lambda_k+h)x+(\lambda_k-h)y\big),
\big((1-\lambda_{k-1}-h)x+(\lambda_{k-1}+h)y\big)\Big)\\
&=\lim_{h\to 0+}\sum_{k=1}^{n+1} L 
\left\|\big((1-\lambda_k+h)x+(\lambda_k-h)y\big)-
\big((1-\lambda_{k-1}-h)x+(\lambda_{k-1}+h)y\big)\right\|\\
&=\sum_{k=1}^{n+1} L 
\left\|z_k-z_{k-1}\right\|=L\|y-x\|\,,
\end{align*}
where we have used the continuity of $f$ and $g$, and that
the intrinsic metric coincides with the Euclidean metric for 
pairs of points for which the connecting line segment has empty intersection
with $\hypsurf$. 

If $s(x,y)\cap \hypsurf$ contains infinitely many points, we can replace $s(x,y)$
by $\gamma$, which is only slightly longer 
than $s(x,y)$,
 but has only finitely
many intersections with $\hypsurf$. A slight modification of the argument 
above then gives that $\|f(y)-f(x)\|<L\|y-x\|+\varepsilon$  for any $\varepsilon>0$, and thus the desired result.
\end{proof}

\begin{conjecture}
 Item \ref{it:conj} of the assumptions of Lemma \ref{th:intrinsic_lip} is not necessary to prove the assertion of the lemma.
\end{conjecture}

We will later give sufficient conditions for item \ref{it:conj} of the assumptions of Lemma \ref{th:intrinsic_lip} to hold, see Lemma \ref{lem:itconj}. These conditions are satisfied in our applications.\\

It is well-known that differentiable functions with bounded derivative
are Lipschitz w.r.t.~the euclidean metric. The same holds true for the
intrinsic metric:

\begin{lemma}\label{lem:diff-lip}
Let $A\subseteq \R^d$ be open and let $f:A\longrightarrow \R^m$ be differentiable with $\|f'\|\le K$.

Then $f$ is intrinsic Lipschitz with constant $K$.
\end{lemma}

\begin{proof}
Let $x,y\in A$ and let $\gamma$ be a continuous curve of finite length
with $\gamma(0)=x$ and $\gamma(1)=y$. (If no such curve exists we trivially
have $\|f(y)-f(x)\|\le K \rho(x,y)=\infty$.)
Let $0=t_0<\ldots<t_n=1$. Without loss of generality the $t_k$ can 
be chosen such that the line segment spanned by $\gamma(t_{k-1})$ and
$\gamma(t_k)$ is in $A$ for every $k$. Then
\begin{align*}
\|f(y)-f(x)\|
&\le\sum_{k=1}^n \|f(\gamma(t_k))-f(\gamma(t_{k-1}))\|\\
&\le\sum_{k=1}^n \sup_{t\in(t_{k-1},t_k)}\|f'(\gamma(t))\|\,\|\gamma(t_k)-\gamma(t_{k-1})\|\\
&\le K \sum_{k=1}^n  \|\gamma(t_k)-\gamma(t_{k-1})\|\le K \ell(\gamma)\,.
\end{align*}

\end{proof}

Furthermore, we prove that the composition of an intrinsic Lipschitz function with a Lipschitz function is intrinsic Lipschitz:

\begin{lemma}\label{lem:int-comp}
Let $A\subseteq \R^d$ be open. 
Let $g:\R^d\longrightarrow A$ be Lipschitz with constant $L_g$.
Let $f:A\longrightarrow \R^m$ be intrinsic Lipschitz with constant $L_f$.

Then $f\circ g$ is intrinsic Lipschitz with constant $L_f L_g$.
\end{lemma}

\begin{proof}
Let $\gamma$ be a continuous curve of finite length with
$\gamma(0)=x$ and $\gamma(1)=y$. (If no such curve exists we trivially
have $\|f(y)-f(x)\|\le L_g \rho(x,y)=\infty$.)
Let 
$0=t_0<\ldots<t_n=1$.
For every $\delta>0$ there are $0=\bar t_0<\ldots<\bar t_{\bar n}=1$ such that
$\rho(g (x),g (y) )<\sum_{k=1}^{\bar n}\|g(\bar t_k)-g(\bar
t_{k-1})\|+\delta/L_f$. So
\begin{align*}
\sum_{k=1}^n\|f\circ g(\gamma(t_k))-f\circ g(\gamma(t_{k-1}))\|
&\le L_f \sum_{k=1}^n\|g(t_k)-g(t_{k-1})\|\\
&\le L_f \rho(g (x),g (y) )\\
&< L_f \left(\sum_{k=1}^{\bar n}\|g(\bar t_k)-g(\bar t_{k-1})\|+\delta/L_f\right)\\
&< L_f \left(L_g\sum_{k=1}^{\bar n}\|\bar t_k-\bar t_{k-1}\|+\delta/L_f\right)\\
&\le L_f L_g \ell(\gamma) + \delta\,.
\end{align*}
Since $\delta>0$ was arbitrary, we obtain the result.
\end{proof}

\subsection{The form of the set of discontinuities}

We are going to generalize the idea of transforming a
discontinuous drift into a Lipschitz one to general dimensions.

For this we assume that the drift coefficient $\mu$ is piecewise Lipschitz
 in the sense of Definition \ref{def:pw-lip}, that is, there 
exists a hypersurface $\hypsurf$ with finitely many components 
such that $\mu|_{\R^d\backslash \hypsurf}$ is intrinsic Lipschitz. 
The assumption on the drift that will make our
method work therefore encompasses assumptions on $\hypsurf$.

\begin{assumption}\label{ass:mu-ddim-sense}\label{ass:hypersurface0}
The drift coefficient $\mu$ is a piecewise Lipschitz function
$\R^d\longrightarrow \R^d$.
Its exceptional set
$\hypsurf$ is a $C^3$ hypersurface.
\end{assumption}

A consequence of Assumption \ref{ass:hypersurface0} is that
locally there exists a $C^2$ orthonormal vector, that is,
for every sufficiently small open and connected  $\subtheta\subseteq\hypsurf$ there exists
an orthonormal vector on $\subtheta$,
i.e.~a $C^2$-function $n:\subtheta\longrightarrow \R^d$ such that
for all $\xi\in \subtheta$ the vector $n(\xi)$ is orthogonal to the tangent
space of $\hypsurf$ in $\xi$ and $\|n(\xi)\|=1$. It is well-known, that there are
in general two possible choices for $n$ and that one can take
$\subtheta=\hypsurf$ only if $\hypsurf$ is orientable. But given $n$ on
$\subtheta$, the only other orthonormal vector is $-n$.

Define the distance $d(x,\hypsurf)$ between a point $x$ and the hypersurface 
$\hypsurf$ in the usual way,
$d(x,\hypsurf):=\inf\{\|x-y\|:y \in \hypsurf\}$.
For every $\varepsilon>0$ we define 
$\hypsurf^\varepsilon:=\{x\in \R^d: d(x,\hypsurf)<\varepsilon\}$.

\begin{assumption}\label{ass:hypersurface1}
There exists $\varepsilon_0>0$ such that 
$\hypsurf^{\varepsilon_0}$ has the {\em unique closest point property}, i.e.~for every $x\in \R^d$
with $d(x,\hypsurf)<\varepsilon_0$ there is a unique $\proj\in \hypsurf$ with
$d(x,\hypsurf)=\|x-\proj\|$.
\end{assumption} 

A set possessing the property described in Assumption \ref{ass:hypersurface1}
is called a {\em set of positive reach}. {\em The reach of a set 
$\hypsurf$ is
the supremum over all $\varepsilon_0>0$ such that 
$\hypsurf^{\varepsilon_0}$ has the {\em unique closest point property}.}
This and the notion of
{\em unique closest point property} 
can be found in \cite{krantz1981}.

\begin{lemma}\label{lem:nd_bounded}
Let $\hypsurf$ be a $C^3$-hypersurface. 

If $\hypsurf$ is of positive reach, then $\|n'\|$ is bounded:
\[
\|n'(\xi)\|\le 2\frac{d-1}{\reach(\hypsurf)}
\]
for all $\xi\in\hypsurf$.
\end{lemma}

The proof of Lemma \ref{lem:nd_bounded} can be found in the appendix.\\

Note that one can find examples of hypersurfaces  
with bounded $\|n'\|$ which are not of positive reach,
see Figure \ref{fig:bounded_ndash}.\\

\begin{figure}[h]
\begin{center}
\includegraphics[scale=0.5]{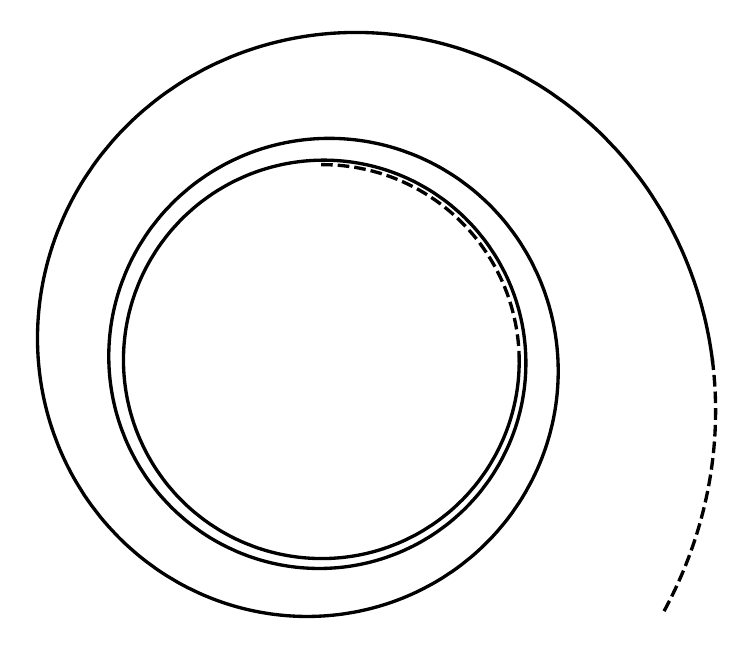}
\end{center}
\caption{A hypersurface in $\R^2$ 
with bounded $\|n'\|$ that is not of positive reach. 
}
\label{fig:bounded_ndash}
\end{figure}

Due to Assumption \ref{ass:hypersurface1} there exists an $\varepsilon_0>0$ 
for which we may define a mapping 
$\proj:\hypsurf^{\varepsilon_0}\longrightarrow \hypsurf$ assigning to each $x$ the point $\proj(x)$ in $\hypsurf$ closest to $x$. \\

\begin{lemma}\label{lem:itconj}
If $\hypsurf$ is a $C^3$-hypersurface that satisfies 
 Assumption
\ref{ass:hypersurface1}, then item \ref{it:conj} of Lemma \ref{th:intrinsic_lip}
is satisfied, i.e.~ for $x,y\in \R^d$ and $\eta>0$ 
there exists a continuous curve $\gamma$ from $x$ to $y$ with $\ell(\gamma)<\|x-y\|+\eta$ such that $\#(\gamma\cap \hypsurf)<\infty$.
\end{lemma}

The rather technical proof of this lemma can be found in the appendix.
Note that for many examples, like a (hyper-)sphere or hyperplane,
 item \ref{it:conj} of 
Lemma \ref{th:intrinsic_lip}
is obviously satisfied. So in these cases there is no need to resort to
Lemma \ref{lem:itconj}. However, it is an interesting fact that this condition
is automatically satisfied under our assumptions on $\hypsurf$.

\subsection{Construction of the transform $G$}
\label{subsec:Constr-ddim}

As before, we construct a transform $G$ with the property that 
the SDE for $G(X)$ has Lipschitz coefficients. 

For this to be well-defined, we make the following assumption:

\begin{assumption}\label{ass:sigma-ddim-sense}
There is a constant $c_0>0$ such that $\|\sigma(\xi)^\top n(\xi)\|\ge c_0$ for all $\xi\in \hypsurf$.
\end{assumption}

\begin{remark}
 Assumption \ref{ass:sigma-ddim-sense} is a {\em non-parallelity condition},
meaning that for all $\xi \in \hypsurf$, $\sigma(\xi)$ must not be parallel to
$\hypsurf$, in the sense that there exists some $x\in\R^d$ such that $\sigma(\xi)x$ is not in the
tangent space of $\hypsurf$ in $\xi$.  \end{remark}

Assumption \ref{ass:sigma-ddim-sense} is by far weaker than uniform
ellipticity. For the practical example we study in Section \ref{sec:Example} it
is satisfied, whereas uniform ellipticity  clearly is not.\\

For defining the transform, we first switch to a local setting. 
Suppose $\tilde
x\in \R^d$ is close to $\hypsurf$, i.e.~$d(\tilde x,\hypsurf)<\varepsilon_0$.
Let $\subtheta\subseteq\hypsurf$ be an open environment of $\proj(x)$ in $\hypsurf$ and 
$n$ an orthonormal vector.
It follows that
the set 
\[
U=\{y_1 n(\xi)+\xi: y_1 \in (-\varepsilon_0,\varepsilon_0),
\xi\in \subtheta\}
\]
is an open environment of $\tilde x$, and every point $x\in U$ can be uniquely
represented in the form $x=y_1 n(\xi)+\xi$,
$y_1\in (-\varepsilon_0,\varepsilon_0)$, $\xi\in \subtheta$.

We are now ready to locally define the transform $G:U\longrightarrow\R^d$ by
\begin{equation}\label{eq:defGhypersurface}
G(x)=
x+\tilde \phi(x) \alpha(\proj(x))\,,
\end{equation}
where $\tilde \phi(x)=(x-\proj(x))\cdot
n(\proj(x))\|x-\proj(x)\|\phi\left(\frac{\|x-\proj(x)\|}{c}\right)$, with $\phi$ (different than in \eqref{eq:bump}) defined by
\begin{align*}
\phi(u)=
\begin{cases}
(1+u)^4(1-u)^4 & \text{if } |u|\le 1\,,\\
0 & \text{else}\,,
\end{cases}
\end{align*}
and where 
 \begin{align}
\label{eq:alphad}\alpha(\xi):=\lim_{h\to 0+}\frac{\mu(\xi-h n(\xi))-\mu(\xi+hn(\xi))}{2 n(\xi)^\top \sigma(\xi)\sigma(\xi)^\top n(\xi)}\,,\qquad\xi\in \subtheta\,.
\end{align}

One important point to note is the following proposition.

\begin{proposition}
The value of the function $G$ does not depend on the choice of the 
orthonormal vector.
\end{proposition}

\begin{proof}
Both $\alpha(\proj(x))$ and $\tilde\phi(x)$ depend on the
parametrization only through the direction of the normal vector $n(\proj(x))$.
But from the definitions of $\tilde\phi$ and $\alpha$  
we see that if $n(\proj(x))$ is replaced by $-n(\proj(x))$, then
$\tilde\phi(x)$ and $\alpha(\proj(x))$ both change sign. Therefore, $\tilde
\phi(x)\alpha(\proj(x))$ does not depend on the particular choice
of the orthonormal vector.
\end{proof}

The only reason why we defined $G$ locally at first was that for a
non-orientable hypersurface we do not have, by definition, a global orthonormal
vector. However, since the value of the locally defined function $G$ does not
depend on the particular choice of the orthonormal vector, we can use the same
equations \eqref{eq:defGhypersurface} and \eqref{eq:alphad} for defining $G$
globally on $\hypsurf^{\varepsilon_0}$. That is, the function $G:\R^d\longrightarrow
\R^d$, 
\[
G(x)=\begin{cases}
 x+\tilde \phi(x) \alpha(\proj(x))&  x\in \hypsurf^{\varepsilon_0}\,,\\
x & x\in \R^d\backslash \hypsurf^{\varepsilon_0}
\end{cases}
\]
is well-defined. Note further that, if we require $c\le\varepsilon_0$, then from
 $d(x,\hypsurf)>\varepsilon_0$ it follows that
$d(x,\hypsurf)>c$ and therefore $\phi(\frac{\|x-p(x)\|}{c})=0$ with a 
$C^2$-smooth paste to $0$ in all points $x$ satisfying $d(x,\theta)=c$.

\subsection{Properties of $G$}

We need to prove the following:
\begin{enumerate}
\item $c$ can be chosen in a way such that  $G$ is a diffeomorphism
 $\R^d\longrightarrow \R^d$;
\item It\^o's formula holds for $G^{-1}$;
\item the SDE for $G(X)$ has Lipschitz coefficients.
\end{enumerate}

\begin{assumption}\label{ass:alpha-ddim}
There is a constant $a$ such that every locally defined function
$\alpha$ as defined in \eqref{eq:alphad} is $C^3$ and all derivatives up to
order 3 are bounded by $a$.  
\end{assumption}

\begin{theorem}\label{thm:Ginv}
Let Assumptions \ref{ass:hypersurface0}--\ref{ass:alpha-ddim} be satisfied.
If the constant $c>0$ appearing in the definition of $\tilde\phi$ is
sufficiently small, then $G$ is a diffeomorphism $\R^d\longrightarrow\R^d$.
\end{theorem}

For proving Theorem \ref{thm:Ginv} we first need to prove two technical lemmas.
For every $\xi\in \hypsurf$, denote by $\tang(\xi)$ the tangent space of $\hypsurf$ in $\xi$.

\begin{lemma}
For $\xi\in \hypsurf$, $n'$ is a linear mapping from $\tang(\xi)$ into $\tang(\xi)$.
\end{lemma}

\begin{proof}
$n'$ is by definition a linear mapping 
$\tang(\xi)\longrightarrow \R^d$.
Furthermore, we have $\|n\|=1$, so that for any curve $\gamma$ in
$\hypsurf$ 
\[
0=\frac{d}{dt}\|n(\gamma(t))\|^2=2 n(\gamma(t))\cdot \big(n'(\gamma(t))\gamma'(t)\big)\,.
\]
If $b\in \tang(\xi)$, we can find a curve $\gamma$ in $\hypsurf$  
such that $\gamma(0)=\xi$ and $\gamma'(0)=b$. Thus, 
$n(\xi)\cdot( n'(\xi)b)=0$, i.e.~$n'(\xi)b\in \tang(\xi)$.
\end{proof}

\begin{remark}\label{rem:invert}
If $\hypsurf$ is $C^3$ and of positive reach $\varepsilon_0$
then we may choose $0<\varepsilon<\varepsilon_0$ such that, 
whenever $y_1\in \R$ with $|y_1|<\varepsilon$, then
$\idt+y_1 n'(\xi)$ is invertible.

Indeed, let $K$ be a bound on $\|n'\|$ and let 
$\varepsilon=\frac{\varepsilon_0}{\varkappa \max(K,1)}$ for some fixed
$\varkappa>1$. Then for $|y_1|< \varepsilon$ we have 
$\|y_1 n'(\xi)\|=|y_1| \|n'(\xi)\|<\frac{1}{\varkappa }<1$,
such that $\idt+y_1 n'(\xi)$ is invertible 
by the subsequent well-known Lemma \ref{th:neumann}.
\end{remark}

\begin{lemma}\label{th:neumann}
Let $\opA$ be a linear operator on a subspace $V\subseteq\R^d$ and let $\opA$ have (operator) norm smaller than $1$.

Then $\idV+\opA$ is invertible and $\|(\idV+\opA)^{-1}\|\le(1-\|\opA\|)^{-1}$.
\end{lemma}

\begin{proof}
Consider the Neumann series
$\opB= \sum_{k=0}^\infty (-\opA)^k $, which converges in operator norm and 
satisfies $\|\opB\|\le (1-\|\opA\|)^{-1}$. Then

\[
(\idV+\opA)\opB
=\sum_{k=0}^\infty (-\opA)^k-\sum_{k=0}^\infty (-\opA)(-\opA)^k
=\sum_{k=0}^\infty (-\opA)^k-\sum_{k=1}^\infty (-\opA)^k=\idV\,.
\]
Thus, $\opB$ is the inverse of $\idV+\opA$.
\end{proof}

\begin{proof}[Proof of Theorem \ref{thm:Ginv}]

Fix some $\varkappa>1$\label{pg:kappa} and set
$\varepsilon=\frac{\varepsilon_0}{\varkappa \max(K,1)}$, where $K$ is a bound on $\|n'\|$, 
which exists by Lemma \ref{lem:nd_bounded}.

Let $0<c<\varepsilon$.

For $\tilde x\notin \hypsurf^c$, differentiability of $G$ in $\tilde x$
is obvious.

For $\tilde x\in \hypsurf^c$ choose an open subset $\subtheta$ of $\hypsurf$ (as before)
and an orthonormal vector $n$ such that $U\subset\R^d$ is an open set with
$U\cap \hypsurf=\subtheta$ and every $x\in U$ can uniquely be written in the form $x=y_1 n(\xi)+\xi$
with  $\xi=\proj(x)$. 
$\hypsurf$ can be parametrized locally by a one-one mapping 
$\psi: R\longrightarrow \R^d$,
where $R\subseteq \R^{d-1}$ is an open rectangle in $\R^{d-1}$, and
there is a point $(\tilde y_2,\dots,\tilde y_d)\in R$ such that 
$\psi(\tilde y_2,\dots,\tilde y_d)=\proj(\tilde x)$. By making $R$ and/or
$\subtheta$ smaller, if necessary, we may w.l.o.g. assume
that $\subtheta=\psi(R)$.

Thus, we have a bijective mapping $\param: (-\varepsilon,\varepsilon)\times R\longrightarrow U$,
\[
\param(y_1,\ldots,y_d):=y_1 n(\psi(y_2,\ldots,y_d))+\psi(y_2,\ldots,y_d)\,,\quad
y\in 
 (-\varepsilon,\varepsilon)\times R\,.
\]
Note that $\proj(\param(y))=\psi(y_2,\ldots,y_d)$ for all $y\in (-\varepsilon,\varepsilon)\times R$.

We have 
\begin{align*}
G\circ \param(y)&
=y_1 n(\psi(y_2,\ldots,y_d))+\psi(y_2,\ldots,y_d)
+y_1|y_1|\phi\left(\frac{|y_1|}{c}\right)\alpha(\psi(y_2,\dots,y_d))\\
&=y_1 n(\psi(y_2,\ldots,y_d))+\psi(y_2,\ldots,y_d)
+\bar\phi(y_1)\alpha(\psi(y_2,\dots,y_d))
\,,
\end{align*}
where $\bar \phi = y |y| \phi \left(\frac{y}{c}\right)$,
and thus
\begin{align*}
\frac{\partial (G\circ \param)}{\partial y_1}(y)&=n(\psi(y_2,\ldots,y_d))+\bar \phi'(y_1)\alpha(\psi(y_2,\dots,y_d))\,,\\ \text{and}\quad
\frac{\partial (G\circ \param)}{\partial y_j}(y)&=y_1\frac{\partial(n\circ \psi)}{\partial y_j}(y_2,\ldots,y_d)+\frac{\partial \psi}{\partial y_j}(y_2,\ldots,y_d)+\bar\phi(y_1)\frac{\partial(\alpha\circ\psi)}{\partial y_j}(y_2,\ldots,y_d)\,.
\end{align*}
Now note that 
\begin{align*}
\frac{\partial (G\circ \param)}{\partial y_1}(y)&=
G'(\param(y)) \frac{\partial \param}{\partial y_1}(y)=G'(\param(y))n(\psi(y_2,\ldots,y_d))\,,\\
\text{and}\quad
\frac{\partial (G\circ \param)}{\partial y_j}(y)&=G'(\param(y))\frac{\partial \param}{\partial y_j}(y)=
G'(\param(y))\left(
y_1\frac{\partial(n\circ \psi)}{\partial y_j}(y_2,\ldots,y_d)+\frac{\partial \psi}{\partial y_j}(y_2,\ldots,y_d)
\right)
\end{align*}
for all $j\ne 1$. Further,
\begin{align*}
\frac{\partial(n\circ \psi)}{\partial y_j}(y_2,\ldots,y_d)&=n'(\psi(y_2,\ldots,y_d))\frac{\partial \psi}{\partial y_j}(y_2,\ldots,y_d)\,,\\
\text{and}\qquad\frac{\partial(\alpha\circ \psi)}{\partial y_j}(y_2,\ldots,y_d)&=\alpha'(\psi(y_2,\ldots,y_d))\frac{\partial \psi}{\partial y_j}(y_2,\ldots,y_d)\,.
\end{align*}
Recall that for any $\xi\in \hypsurf$, we have that $n'(\xi)$ and $\alpha'(\xi)$ are 
linear mappings from the tangent space of $\hypsurf$ in $\xi$ into the $\R^d$.
For $\xi=\psi(y_2,\dots,y_d)$  
it then follows that 
\[
G'(\param(y))\left(
\idt+y_1 n'(\xi) 
\right)\frac{\partial \psi}{\partial y_j}(y_2,\ldots,y_d)
=\left(\idt+y_1n'(\xi)+\bar \phi(y_1)\alpha'(\xi)\right)\frac{\partial \psi}{\partial y_j}(y_2,\ldots,y_d)\,.
\]
Since this equation holds for all $\frac{\partial \psi}{\partial y_j}$, $j=2,\dots,d$, it also holds for every vector $b$ in 
the tangent space, i.e.
\[
G'(\param(y))\left(
\idt+y_1 n'(\xi) 
\right)\vv
=\left(\idt+y_1n'(\xi)+\bar \phi(y_1)\alpha'(\xi)\right)\vv\,.
\]
For $|y_1|\le \varepsilon$, the 
mapping $\idt+y_1n'(\xi)$ is invertible by the argument from Remark 
\ref{rem:invert}. Denote the
inverse of $\idt+y_1n'(\xi)$ by $\mm(y)$. 

Then for any $b\in \tang(\xi)$ we can write
$\vv=\big(\idt+y_1n'(\xi)\big)\vv_1$ 
with $\vv_1=\mm(y) b$ and therefore
\[
G'(\param(y))\vv
=\vv+\bar \phi(y_1)\alpha'(\xi)\mm(y)
\vv\,.
\]
For a general vector $\vv\in \R^d$ we have that $(\vv\cdot n)n=n n^\top \vv$
is orthogonal to the tangent space and $(\id-nn^\top)\vv$ is in the tangent space.

We abbreviate $G'=G'(\tilde x)$, $\proj=\proj(\tilde x)$, 
$d=\|\tilde x-\proj(\tilde x)\|$, $n=n(\proj(\tilde x))$, 
$n'=n'(\proj(\tilde x))$, $\mm=\mm(\param^{-1}(\tilde x))$.
Then we have for $\vv\in \R^d$
\begin{align*}
G' \vv &= G' \Big((\vv\cdot n)n+(\vv-(\vv\cdot n)n)\Big)\\
&= (\vv\cdot n)G' n+G'(\vv-(\vv\cdot n)n)\\
&=(\vv\cdot n)(n+\bar\phi'(d)\alpha(\proj))+(\vv-(\vv\cdot n)n)+\bar\phi(d)\alpha'(p)\mm (\vv-(\vv\cdot n)n)\\
&=\vv+\bar\phi'(d)\alpha(\proj)n^\top\vv+\bar\phi(d)\alpha'(p)\mm(\id-nn^\top)\vv\,.
\end{align*}
Therefore,
\[
G'=\id +\bar\phi'(d)\alpha(\proj)n^\top+\bar\phi(d)\alpha'(p)\mm(\id-nn^\top)\,,
\]
or, more explicitly,
\begin{align}\label{eqn:neumannG}
G'(\tilde x)=&\id +\bar\phi'(\|\tilde x-\proj(\tilde x)\|)\alpha(\proj(\tilde x))n(\proj(\tilde x))^\top\nonumber\\
&+\bar\phi(\|\tilde x-\proj(\tilde x)\|)\alpha'(\proj(\tilde x))\mm(\param^{-1}(\tilde x))(\id-n(\proj(\tilde x))n(\proj(\tilde x))^\top)\,.
\end{align}

In order to apply Hadamard's global inverse function theorem \cite[Theorem 2.2]{ruzhansky2015}
and thus to show
that $G$ is a diffeomorphism $\R^d\longrightarrow\R^d$, we need to show
that $G$ is $C^1$, $G'(x)$ is invertible for all $x\in \R^d$, and 
$\lim_{\|x\|\to \infty} \|G(x)\|=\infty$.

We have already proven differentiability of $G$ in $\tilde x$.
If $c$ is sufficiently small, $G'(\tilde x)$ is invertible, since 
$\bar\phi'$ and $\bar\phi$ are uniformly bounded with a bound that tends
to 0 for $c \to 0$.   For $c$ small enough it is therefore guaranteed that
$G'(\tilde x)$ is close to the identity and therefore invertible by 
Lemma \ref{th:neumann}.
We show in the separate Lemma \ref{lem:c-ddim} that $c>0$ can be chosen uniformly for all $\tilde x$ such that $G'(\tilde x)$ is invertible.

Since $G(x)=x+\bar\phi(x)\alpha(x)$ and both $\bar\phi$ and $\alpha$ are
bounded by the definition of $\bar\phi$ and Assumption \ref{ass:alpha-ddim}, we also have the third requirement of Hadamard's global inverse
function theorem. $G$ is therefore a diffeomorphism.
\end{proof}

We will see that $c$ can always be chosen sufficiently small in the proof of Theorem \ref{thm:Ginv}.

\begin{lemma}\label{lem:c-ddim}    
Assume Assumptions \ref{ass:mu-ddim-sense} -- \ref{ass:alpha-ddim}. 
Fix  $\varkappa>1$
and let 
\[
c_0:=\min\left(1,\tfrac{\varepsilon_0}{\varkappa\max(K,1)},\left(1+\frac{d}{3}  \sup_{\xi\in\hypsurf}\left(\max_{1\le i\le d}|\alpha_i(\xi)|+\frac{d}{4} \frac{\varkappa}{\varkappa-1}\max_{1\le i,j\le d}\left|\tfrac{\partial \alpha_i(\xi)}{\partial x_j}\right| \right)\right)^{-1}\right)\,.
\]
Then for every choice of $c\in (0,c_0)$ we have that $G'(x)$ is
invertible for every $x\in \R^d$.
\end{lemma}

\begin{proof}
Note that $c_0>0$. 

Let $x\in \R^d$ and recall equation \eqref{eqn:neumannG} from the proof of Theorem \ref{thm:Ginv}\begin{align*}
 G'(x)&=\id + \bar \phi'(\|x-p(x)\|) \alpha(p(x)) n(p(x))^\top\\ &+ \bar \phi(\|x-p(x)\|) \alpha'(p(x)) \mm(\param^{-1}(x))(\id-n(p(x)) n(p(x))^\top)=:1+\opA(x)\,.
\end{align*}
We begin by estimating the operator norm of $\opA(x)$ for given $c\in (0,c_0)$.
\begin{align*}
\|\opA(x)\| 
&\le \| \bar \phi'(\|x-p(x)\|)\|\  d \max_{1\le i\le d}  |\alpha_i(p(x))|  \\
&+\bar \phi(\|x-p(x)\|)  \|\mm\| \|\id-n(p(x)) n(p(x))^\top\|\ d^2\! \max_{1\le i,j\le d}\left|\frac{\partial \alpha_i(p(x))}{\partial x_j}\right|
\\
&\le \frac{cd}{3} \max_{1\le i\le d} |\alpha_i(p(x))|+\frac{c^2d^2}{12} \frac{1}{1-|y_1| \|n'\|}\max_{1\le i,j\le d}\left|\frac{\partial \alpha_i(p(x))}{\partial x_j}\right| \,,
\end{align*}
where we used that $\|\bar \phi'(\|x-p(x)\|)\| \le \frac{c}{3}$ and
$|\bar \phi(\|x-p(x)\|)|\le \frac{c^2}{12}$ for $x\in \hypsurf^c$ (by estimating the maxima), and that
$\|\id-n(p(x)) n(p(x))^\top\|\le 1$.
Furthermore $\|\mm\| \le \frac{1}{1-|y_1| \|n'\|}$, since $\|y_1 n'\|
<\frac{1}{\varkappa}<1$ by $c<\frac{\varepsilon_0}{\varkappa \max(K,1)}$, 
Lemma \ref{th:neumann} and Remark \ref{rem:invert}.  Hence 
\[
 \frac{1}{1-|y_1| \|n'\|} \le \frac{\varkappa}{\varkappa-1}\,.
\]
Therefore $\|\opA(x)\| \le \frac{cd}{3}  \left(\max_{1\le i\le d}|\alpha_i(p(x))|+\frac{cd}{4} \frac{\varkappa}{\varkappa-1}\max_{1\le i,j\le d}\left|\frac{\partial \alpha_i(p(x))}{\partial x_j}\right| \right)$.

We want $c$ small enough to have 
\(\|\opA(x)\| <1\, \)
and to that end we choose $c<1$ and \[
c<\left(1+\frac{d}{3}  \left(\max_{1\le i\le d}|\alpha_i(p(x))|+\frac{d}{4}
\frac{\varkappa}{\varkappa-1}\max_{1\le i,j\le d}\left|\frac{\partial
\alpha_i(p(x))}{\partial x_j}\right| \right)\right)^{-1}\,.\]
Hence  
$G'(x)$ is invertible
 for $x\in \hypsurf^c$ by Lemma \ref{th:neumann}.
For $ x\in \R^d\backslash \hypsurf^c$, $G'(x)=\id$.
\end{proof}

W.l.o.g. we always choose $c$ like in Lemma \ref{lem:c-ddim}.\\

We proceed with proving that, although $G \notin C^2$, It\^o's formula holds for
$G$ and $G^{-1}$.

\begin{theorem}\label{th:ItoGi}
Let Assumptions \ref{ass:hypersurface0}--\ref{ass:alpha-ddim} be satisfied.

Then It\^o's formula holds for $G$ and $G^{-1}$.
\end{theorem}

\begin{proof}
If $x\in \R^d\backslash \hypsurf$, then since $G,G^{-1} \in C^2$ on $\R^d\backslash \hypsurf$, It\^o's formula holds for
$G$ and $G^{-1}$ until the first
time $X$ hits $\hypsurf$. So the only interesting case is $x\in \hypsurf$.

For this, there exists an open rectangle $R\in \R^{d-1}$ and a local
parametrization $\psi:R\longrightarrow \R^{d}$ of $\hypsurf$. Let $\subtheta=\psi(R)$. Moreover, 
\[
U=\{y_1
n(\psi(y_2,\ldots,y_d))+\psi(y_2,\ldots,y_d):y_1\in (-\varepsilon,\varepsilon), (y_2,\ldots,y_d)\in R\}\,.
\]
Let $\param:(-\varepsilon,\varepsilon)\times R\longrightarrow U$ be defined
as in the proof of Theorem \ref{thm:Ginv}.
Note that $\param\in C^2$, because $\hypsurf$ is $C^3$ by Assumption \ref{ass:hypersurface0}, so It\^o's formula holds
for $\param$. 
$\param$ is locally invertible with $\det \param'\ne 0$, so $\param^{-1}\in C^2$ as well. 
If we can show that It\^o's formula holds for $G\circ \param$, then it
also holds  for $G=G\circ \param \circ \param^{-1}$. 

$G\circ \param$ fits the assumptions of \cite[Theorem  2.9]{sz14} 
(we get boundedness of the derivatives by localizing to a bounded domain),
so It\^o's formula holds for $G\circ \param$, and therefore also for $G$.\\

$\tilde G=\param^{-1}\circ G\circ \param$ is a function with continuous first and second 
derivatives, with the sole exception of $\frac{\partial^2 \tilde G}{\partial
y_1^2}$, which is bounded, but may be discontinuous for $y_1=0$.  
Since $\det \tilde G'\ne 0$ on an environment of $x$,  this property 
transfers to the inverse, which is $\tilde G^{-1}=\param^{-1}\circ G^{-1}\circ \param$. 
Thus, again by \cite[Theorem  2.9]{sz14}, It\^o's formula holds for 
$\tilde G^{-1}$, and a fortiori for $G^{-1}$.
\end{proof}

Now we are ready to show that the coefficients of the transformed SDE for $G(X)$ are Lipschitz.

\begin{assumption}\label{ass:mutilde-lipschitz}
We assume the following for $\mu$ and $\sigma$:
\begin{enumerate}
\item\label{it:mutilde-lipschitz1} the diffusion coefficient $\sigma$ is Lipschitz;
\item\label{it:mutilde-lipschitz2} $\mu$ and $\sigma$ are bounded on $\hypsurf^\varepsilon$.
\end{enumerate}
\end{assumption}

\begin{assumption}\label{ass:C4n}
The exceptional set $\hypsurf$  of $\mu$ is $C^4$.
Every unit normal vector $n$ of $\hypsurf$ has bounded second and third derivative.
\end{assumption}

\setcounter{newlemma}{25}
\begin{newlemma}\label{lem1}
Assume Assumptions \ref{ass:hypersurface0}, \ref{ass:hypersurface1}, and \ref{ass:C4n}.
Let $c\in (0,\varepsilon_0)$. 

Then the function $\tilde \phi\colon\Theta^c\setminus\Theta\to \R$
with $\tilde \phi(x)=(x-\proj(x))\cdot n(\proj(x))\|x-\proj(x)\|\phi\left(\frac{\|x-\proj(x)\|}{c}\right)$
is three times differentiable with bounded first, second, and third derivative.
\end{newlemma}

\begin{proof}
For $x\in \hypsurf^c\setminus \hypsurf$ we have $(x-\proj(x))\cdot
n(\proj(x))\|x-\proj(x)\|=s\, d(x,\hypsurf)^2$ with $s\in \{-1,1\}$. 
By \cite[Corollary 4.5]{dud
}, $d(\cdot,\hypsurf)$ is $C^4$ on $\hypsurf^c\setminus \hypsurf$.

Since $p'(x)$ maps into the tangent space of $\Theta$ in $p(x)$ it holds that
$(x-p(x))^\top p'(x)=0$.  Thus we have
$(d(x,\Theta)^2)'=(\|x-p(x)\|^2)'=2(x-p(x))^\top (\id-p'(x))=2(x-p(x))^\top$.
Note that $(x-p(x))^\top$ is bounded by $c$ on $\hypsurf^c$.

The function $p\colon\Theta^c\to \Theta$ is $C^3$ by Assumptions
\ref{ass:hypersurface0}, \ref{ass:hypersurface1}, and \ref{ass:C4n} and
\cite[Theorem 4.1]{dudekholly}.  

By
Assumptions \ref{ass:hypersurface0}, \ref{ass:hypersurface1} and Lemma
\ref{lem:nd_bounded} the first derivative of every unit normal vector $n$ is
bounded, and by Assumption \ref{ass:C4n} the second and third derivative of
$n$ are bounded. 
Now \cite[Corollary 4]{leostein2018} implies that $p'$, $p''$, and $p'''$ are bounded on $\Theta^c$.

Now it follows from the chain and product rule 
that the function $x\mapsto d(x,\hypsurf)^2$ and its 
derivatives up to order 4 are bounded on $\Theta^c\setminus \Theta$.


Note further that 
\[\phi\left(\frac{\|x-\proj(x)\|}{c}\right)
=\begin{cases}
\left(1-\tfrac{d(x,\hypsurf)^2}{c^2}\right)^4& d(x,\hypsurf)<c\,,\\
0& \text{ else}\,.
\end{cases}
\]
In total, by the chain and product rule, the first three derivatives of  $\tilde \phi$ are bounded.
\end{proof}

\begin{newlemma}\label{lem2}
Assume Assumptions \ref{ass:hypersurface0}, \ref{ass:hypersurface1}, \ref{ass:alpha-ddim}, and \ref{ass:C4n}.
Let $c\in (0,\varepsilon_0)$. 

Then the function $\alpha\circ p\colon\Theta^c\setminus\Theta\to \R^d$
is three times differentiable with bounded first, second, and third 
derivative.
\end{newlemma}
\begin{proof}
By Assumption \ref{ass:alpha-ddim}, $\alpha$ is three times differentiable with 
bounded first, second, and third derivative. 
As shown in the proof of Lemma \ref{lem1}, $p\colon\Theta^c\to \Theta$ is $C^3$ and $p'$, $p''$, and $p'''$ are bounded on $\Theta^c$.  
The chain and product rules now assure that $(\alpha\circ p)'$, $(\alpha\circ p)''$, $(\alpha\circ p)'''$
are bounded.
\end{proof}

\begin{newlemma}\label{lem3}
Let Assumptions \ref{ass:hypersurface0}--\ref{ass:C4n} be satisfied.
Then $G''$ is bounded and it is differentiable with bounded derivative on $\hypsurf^c\backslash \hypsurf$.
\end{newlemma}

\begin{proof}
A sufficient condition for this is, by the definition
of $G$ and the product rule, that the functions $x\mapsto \tilde \phi(x)$
and $x\mapsto \alpha(p(x))$ have this property.
This is guaranteed by Lemmas \ref{lem1} and \ref{lem2}.
\end{proof}

\begin{theorem}\label{th:SDElip-surf}
Let Assumptions \ref{ass:hypersurface0}--\ref{ass:C4n} be satisfied.

Then the SDE for $G(X)$ has Lipschitz coefficients.
\end{theorem}

\begin{proof}
We first show that the drift of $G(X)$ is
continuous in $\hypsurf$.
Let $B$, $R$, $\psi$, and $\param$ be defined as in the proof of 
Theorem \ref{thm:Ginv}. 
Suppose now, we have a locally defined process $X$ in $U$.
Then there exists a locally defined process $Y$ in $(-\varepsilon,\varepsilon)\times R$ with
\[
X=Y_1 n(\psi(Y_2,\ldots,Y_d))+\psi(Y_2,\ldots,Y_d)\,,
\]
i.e.~$X=\param(Y)$.
 
If $Y$ is a locally defined solution to 
\(dY=\nu(Y)dt+\omega(Y) dW\), then by It\^o's formula
\[
dX=\param'(Y)\nu(Y)dt + \param'(Y) \omega(Y) dW + \frac{1}{2}\tr(\omega^\top(Y) \param''(Y) \omega(Y))dt\,,
\]
where $\param'$ and $\param''$ denote the Jacobian and the Hessian of $\param$,
and $\tr$ denotes the trace of a matrix. We want $\param'\omega=\sigma$, or more precisely $\param'(Y)\omega(Y)=\sigma(\param(Y))$, i.e.~$\omega=(\param')^{-1}\sigma$. 
For brevity write $\paraminv=\param^{-1}$. Now
\begin{align*}
(\omega \omega^\top)_{1,1}&=\omega_{1,1}^2+\dots+\omega_{1,d}^2=e_1^\top \omega \omega^\top e_1
=e_1^\top \left(\paraminv'\sigma\sigma^\top\left(\paraminv'\right)^\top\right)e_1\,.
\end{align*}
We show that $\left(\paraminv'\right)^\top e_1=n$.
It is not hard to see that the Jacobian $\param'$ of $\param$ in a point $\xi \in \hypsurf$ is given by
\[
\param'=\left(n, \frac{\partial \psi}{\partial y_2},\dots,\frac{\partial \psi}{\partial y_d}\right)\,,
\]
such that 
\begin{align*}
e_1^\top (\param')^{-1}=e_1^\top\left((\param')^{-1}\right)=n^\top\;
\Longleftrightarrow\; e_1^\top=n^\top \param' =n^\top\left(n, \frac{\partial \psi}{\partial y_2},\dots,\frac{\partial \psi}{\partial y_d}\right)  = (\|n\|^2,0,\dots,0)=e_1^\top\,.
\end{align*}
Therefore we have 
$\omega_{1,1}^2+\dots+\omega_{1,d}^2=n^\top\sigma\sigma^\top n$ on $\hypsurf$.\\

The drift coefficient $\nu$ of the SDE for $Y$ has only discontinuities 
in the set $\{y\in \R^d:y_1=0\}$.
Further,
\[
dY=d(\paraminv(X))=\paraminv'(X)\mu(X)dt+\paraminv'(X)\sigma(X)dW+\frac{1}{2}\tr\left(\sigma^\top(X) \paraminv''(X) \sigma(X)\right)dt\,,
\]
i.e.~$\nu(y)=\paraminv'(\param(y))\mu(\param(y))+\frac{1}{2}\tr\left(\sigma^\top(\param(y)) \paraminv''(\param(y)) \sigma(\param(y))\right)$. The second term is continuous, so that 
\begin{align}
\lefteqn{\lim_{h\to 0+}(\nu(-h,y_2,\ldots,y_d)-\nu(h,y_2,\ldots,y_d)) }\nonumber\\
&=\paraminv'(\param(0,y_2,\ldots,y_d))\lim_{h\to 0+}\left(\mu(\param(-h,y_2,\ldots,y_d))-\mu(\param(h,y_2,\ldots,y_d))\right)\nonumber{}\\
&=\paraminv'(\param(0,y_2,\ldots,y_d))\lim_{h\to 0+}\Big(\mu\big(\param(0,y_2,\ldots,y_d)-hn(\param(0,y_2,\ldots,y_d))\big)\nonumber{}\\
&\hspace{13em}-\mu\big(\param(0,y_2,\ldots,y_d)+hn(\param(0,y_2,\ldots,y_d))\big)\Big)\nonumber\\
&=\paraminv'(\param(0,y_2,\ldots,y_d))2\alpha(\param(0,y_2,\ldots,y_d))(n^\top\sigma\sigma^\top n)(\param(0,y_2,\ldots,y_d))\nonumber{}\\
&=\paraminv'(\param(0,y_2,\ldots,y_d))2\alpha(\param(0,y_2,\ldots,y_d))(\omega\omega^\top)_{11}(0,y_2,\ldots,y_d)\label{eq:nudiff}\,.
\end{align}
Consider 
\begin{align*}
(G\circ \param)(y)&=\param(y)+\tilde \phi(\param(y))\alpha(\proj(\param(y)))\\
& = y_1 n(\param(0,y_2,\ldots,y_d))+\param(0,y_2,\ldots,y_d)+
y_1|y_1|\phi\left(\frac{y_1}{c}\right)\alpha\left(\param(0,y_2,\ldots,y_d)\right)\,,
\end{align*}
and  
\begin{align*}
(\paraminv\circ G\circ \param)(y)
&=\paraminv\left( y_1 n(\param(0,y_2,\ldots,y_d))+\param(0,y_2,\ldots,y_d)+
y_1|y_1|\phi\left(\frac{y_1}{c}\right)\alpha\left(\param(0,y_2,\ldots,y_d)\right)\right)\,.
\end{align*}
Differentiation yields
\begin{align*}
\frac{\partial}{\partial y_1}(\paraminv\circ G\circ \param)(y)
&=\paraminv'\left((G\circ \param)(y) \right)\frac{\partial}{\partial y_1}\Big( y_1 n(\param(0,y_2,\ldots,y_d))+\param(0,y_2,\ldots,y_d)\\
&\hspace{12em}+ y_1|y_1|\phi\left(\frac{y_1}{c}\right)\alpha\left(\param(0,y_2,\ldots,y_d)\right)\Big) \\
&=\paraminv'\left((G\circ \param)(y) \right)\Big(  n(\param(0,y_2,\ldots,y_d))\\
&\hspace{9em}+
\Big(2|y_1|\phi\left(\frac{y_1}{c}\right)+
y_1|y_1|\phi'\left(\frac{y_1}{c}\right)\frac{1}{c}\Big)\alpha\left(\param(0,y_2,\ldots,y_d)\right)\Big)\,.
\end{align*}
We look at the second derivative w.r.t. $y_1$:
\begin{align*}
\frac{\partial^2}{\partial y_1^2}(\paraminv\circ G\circ \param)(y)
&=\text{something continuous}+\paraminv'\left((G\circ \param)(y) \right)\left(  2\sign(y_1)\phi\left(\frac{y_1}{c}\right)\alpha\left(\param(0,y_2,\ldots,y_d)\right)\right)\,.
\end{align*}
Since $G(x)=x$ for $x\in \hypsurf$, we have that $G(\param(y))=\param(y)$ for
$y_1=0$, and thus 
\begin{align}
\lefteqn{\lim_{h\to 0+}\left(\frac{\partial^2}{\partial y_1^2}(\paraminv\circ G\circ \param)(-h,y_2,\ldots,y_d)
-\frac{\partial^2}{\partial y_1^2}(\paraminv\circ G\circ \param)(h,y_2,\ldots,y_d)\right)}\nonumber{}\\
&\qquad=-4\paraminv'\left((G\circ \param)(0,y_2,\ldots,y_d) \right) \alpha\left(\param(0,y_2,\ldots,y_d)\right)\nonumber\\
&\qquad=-4\paraminv'\left(\param(0,y_2,\ldots,y_d) \right) \alpha\left(\param(0,y_2,\ldots,y_d)\right)\,.\hspace{10em}
\label{eq:dy12}
\end{align}
Consider the drift coefficient of $(\paraminv\circ G\circ \param)_k(Y)$, which is
\begin{align*}
\tilde \nu_k(y):=\sum_{j=1}^d \frac{\partial}{\partial y_j}(\paraminv\circ G\circ \param)_k(y) \nu_j(y) 
+\frac{1}{2}\sum_{i,j=1}^d\frac{\partial^2}{\partial y_i\partial y_j}(\paraminv\circ G\circ \param)_k(y)\sum_{l=1}^d\omega_{li}(y)\omega_{lj}(y)\,.
\end{align*}
$(\paraminv\circ G\circ \param)'(0,y_2,\dots,y_d)=\id$, thus 
$\frac{\partial}{\partial y_j}(\paraminv\circ G\circ \param)_k(0,y_2,\dots,y_d)=(e_k)_j$.
Further, note that $\frac{\partial^2}{\partial y_i\partial y_j}(\paraminv\circ
G\circ \param)_k$ is continuous for all pairs $(i,j)$ except $(i,j)=(1,1)$.

Thus, using \eqref{eq:nudiff} and \eqref{eq:dy12},  we have
\begin{align*}
\lefteqn{\lim_{h\to 0+} \left(\tilde \nu_k(-h,y_2,\dots,y_d)
-\tilde \nu_k(h,y_2,\dots,y_d)\right)}\\
&=\lim_{h\to 0+}\Big(
\nu_k(-h,y_2,\dots,y_d)+\frac{1}{2}\frac{\partial^2}{\partial y_1^2}(\paraminv\circ G\circ \param)_k(-h,y_2,\dots,y_d)(\omega\omega^\top)_{11}(0,y_2,\dots,y_d)\\
&\qquad\qquad-\nu_k(h,y_2,\dots,y_d)-\frac{1}{2}\frac{\partial^2}{\partial y_1^2}(\paraminv\circ G\circ \param)_k(h,y_2,\dots,y_d)(\omega\omega^\top)_{11}(0,y_2,\dots,y_d)\Big)\\
&= \paraminv'(\param(0,y_2,\ldots,y_d))2\alpha(\param(0,y_2,\ldots,y_d))(\omega\omega^\top)_{11}(0,y_2,\ldots,y_d)\\&\quad-2\paraminv'\left(\param(0,y_2,\ldots,y_d) \right) \alpha\left(\param(0,y_2,\ldots,y_d)\right)(\omega\omega^\top)_{11}(0,y_2,\ldots,y_d)\;=\;0\,.
\end{align*}
Therefore $\tilde \nu$ is continuous on the whole of $\R^d$.

Now the drift coefficient of the SDE for the process $G(X)$ is continuous as
well: $G(X)=\param\circ (\paraminv\circ G\circ \param)\circ \paraminv
(X)$ and compounding with $\param$ and $\paraminv$ preserves continuity 
of the drift since $\param,\paraminv\in C^2$.

The $k$-th coordinate of the transformed drift $\tilde \mu$ has the form 
\begin{align*}
\tilde \mu_k(z)&=G_k'(G^{-1}(z))\mu(G^{-1}(z))+\frac{1}{2}\tr\left(\sigma^\top(G^{-1}(z))G_k''(G^{-1}(z))\sigma(G^{-1}(z))\right)
\end{align*}
and we have just seen that it is continuous in all $z\in \hypsurf$. 
It remains to show that $\tilde \mu$ is intrinsic Lipschitz on 
$\R^d\backslash\hypsurf$. For $z\in\R^d\backslash\hypsurf^c$ we have 
$\tilde \mu(z)=\mu(z)$.  $\mu$ is intrinsic Lipschitz on
$\R^d\backslash\hypsurf$, and therefore also on $\R^d\backslash\hypsurf^c$.

On $\hypsurf^c\backslash \hypsurf$ we have that $G'$ is differentiable with 
bounded derivative and is therefore intrinsic Lipschitz by Lemma
\ref{lem:diff-lip}. $\mu$ is intrinsic Lipschitz on $\R^d\backslash \hypsurf$
by Assumption \ref{ass:mu-ddim-sense} and $\mu$ is bounded on $\hypsurf^c$ by
Assumption \ref{ass:mutilde-lipschitz}, item \ref{it:mutilde-lipschitz2}. Moreover, $G^{-1}$ is Lipschitz on $\R^d$ and thus 
the mapping $x\mapsto G_k'(G^{-1}(z))\mu(G^{-1}(z))$ is intrinsic Lipschitz by Lemma \ref{lem:int-comp}.

In the same way we see that
$G''$ is differentiable with 
bounded derivative on $\hypsurf^c\backslash \hypsurf$ and is therefore intrinsic
Lipschitz by Lemma \ref{lem:diff-lip}. 
$\sigma$ is Lipschitz on $\R^d$ and therefore intrinsic Lipschitz on 
$\hypsurf^c\backslash \hypsurf$. Moreover, both $G''$ and $\sigma$ are bounded on
 $\hypsurf^c\backslash \hypsurf$, thus $z\mapsto
\frac{1}{2}\tr\left(\sigma^\top(G^{-1}(z))G_k''(G^{-1}(z))\sigma(G^{-1}(z))\right)$
is intrinsic Lipschitz by Lemma \ref{lem:int-comp}.

Now $\tilde \mu$ is intrinsic Lipschitz as a sum of intrinsic Lipschitz functions.

Altogether we have shown that $\tilde \mu$ is piecewise Lipschitz and
continuous, and hence Lipschitz by Lemma \ref{th:intrinsic_lip} and Lemma \ref{lem:itconj}.\\

The transformed diffusion coefficient is given by
\begin{align*}
\tilde \sigma(z)&=G'(G^{-1}(z))\sigma(G^{-1}(z))\,.
\end{align*}
Since $G^{-1}$, $G'$ and $\sigma$ are Lipschitz, the mappings 
$z\mapsto G'(G^{-1}(z))$ and $z\mapsto \sigma(G^{-1}(z))$ are Lipschitz.
Moreover, they are both bounded on $\hypsurf^\varepsilon$ (and thus on $\hypsurf^c$), such that their product is Lipschitz.
\end{proof}

\subsection{Main results} 

Finally, we are ready to prove the two main results of this paper.

For this, define
\begin{align}\label{eq:Z-ddim}
 dZ=dG(X)=\tilde \mu(Z) dt + \tilde \sigma(Z) dW\,, \qquad Z_0=G(x)\,,
\end{align}
where $\tilde \mu$ and $\tilde \sigma$ are defined in the proof of Theorem \ref{th:SDElip-surf}.

\begin{theorem}\label{exun-dd}
Let Assumptions \ref{ass:hypersurface0}--\ref{ass:C4n} be satisfied.

Then the $d$-dimensional SDE \eqref{eq:SDE} has a unique global strong solution.
\end{theorem}

\begin{proof}
 Since by Theorem \ref{th:SDElip-surf} SDE \eqref{eq:Z-ddim} has Lipschitz coefficients, it follows that it has a unique global strong solution
for the initial value $G(x)$.  Due to Theorem
\ref{thm:Ginv}, the transformation $G$ has a global inverse $G^{-1}$. It\^o's formula holds for $G^{-1}$ by Theorem \ref{th:ItoGi}.  Applying It\^o's
formula to $G^{-1}$, we obtain that $G^{-1}(Z)$ satisfies
\begin{align*}
dX=\mu(X)dt+\sigma(X)dW\,,\quad X_0=x\,.
\end{align*}
Setting $X=G^{-1}(Z)$ closes the proof.
\end{proof}

For calculating the solution to the $d$-dimensional SDE \eqref{eq:SDE}, the same algorithm as for the one-dimensional case works,
if applied using the transformations from the $d$-dimensional case.
Let $Z_T^{(\delta)}$ be the Euler-Maruyama approximation of the solution to SDE \eqref{eq:Z-ddim} with step size smaller than $\delta>0$.
\begin{algorithm}\label{alg:num-ddim}
Go through the following steps:
\begin{enumerate}
 \item Set $Z^{(\delta)}_0=G(x)$.
 \item Apply the Euler-Maruyama method to SDE \eqref{eq:Z-ddim} to obtain $Z_T^{(\delta)}$.
 \item Set $\bar X = G^{-1}\left(Z_T^{(\delta)}\right)$.
\end{enumerate}
\end{algorithm}

\begin{theorem}\label{thm:conv-dd}
Let Assumptions \ref{ass:hypersurface0}--\ref{ass:C4n} be satisfied.

Then Algorithm \ref{alg:num-ddim}
converges with strong order $1/2$ to the solution X of the $d$-dimensional SDE \eqref{eq:SDE}.
\end{theorem}

\begin{proof}
We estimate the $L^2$-error of the approximation.
For every $T>0$ there is a constant $C$, such that 
 \begin{align*}
\E\left(\left\|X_T- \bar X_T\right\|^2\right)
= \E\left(\left\|G^{-1}\left(Z_T\right)- G^{-1}\left(Z_T^{(\delta)}\right)\right\|^2\right)
\le L_{G^{-1}}^2\E\left(\left\|Z_T- Z_T^{(\delta)}\right\|^2\right)= L_{G^{-1}}^2C \delta
\end{align*}
for every sufficiently small step size $\delta$, where $L_{G^{-1}}$ is the Lipschitz constant of $G^{-1}$.
We used \cite[Theorem 10.2.2]{kloeden1992} for the $L^2$-convergence of order $1/2$
of the Euler-Maruyama scheme for SDEs with Lipschitz coefficients.
\end{proof}


\subsection{Compact set of discontinuities}
\label{subsec:compact}

To be able to prove our main results we had to make a number of assumptions
on the coefficient functions $\mu$ and $\sigma$. At least one of those 
is indispensable for our method to work, that is, 
Assumption \ref{ass:mu-ddim-sense}, which demands
that $\mu$ is piecewise Lipschitz and that its set of discontinuities $\hypsurf$
is a $C^3$ hypersurface.

There are two more assumptions on $\hypsurf$ and several on the behaviour of the 
coefficients close to $\hypsurf$. In this subsection we shall find out which
assumptions are automatically satisfied in the case where $\hypsurf$ is
compact.

For compact $\hypsurf$, Assumption \ref{ass:hypersurface1} is automatically
satisfied. This follows from a
lemma in \cite{foote1984}:

\begin{lemma}
Let $\hypsurf\subseteq\R^d$ be a compact $C^k$ submanifold with $k\ge 2$. 

Then $\hypsurf$ has a neighbourhood $U=\hypsurf^\varepsilon$ with the unique closest point property,
and the projection map $p:U \longrightarrow \hypsurf$ is $C^{k-1}$.
\end{lemma}

Assumption \ref{ass:sigma-ddim-sense} prescribes a certain geometrical 
relation between $\hypsurf$ and directions of the diffusion coefficient.
This will not be satisfied automatically only from making additional 
assumptions on $\hypsurf$, of course. But for the case of compact $\hypsurf$, 
 Assumption \ref{ass:sigma-ddim-sense} follows easily from weaker requirements
on $\sigma$.

\begin{proposition}\label{th:weaker-thing}
Let $\hypsurf$ be a compact $C^2$ hypersurface and let $\sigma:\R^d\rightarrow \R^{d\times d}$ be Lipschitz. 

If $\sigma(\xi)^\top n(\xi)\ne 0$ for all $\xi\in \hypsurf$, then there exists a
constant $c_0>0$ such that $\|\sigma^\top(\xi)n(\xi)\|\ge c_0$ for all $\xi\in
\hypsurf$.  
\end{proposition}

\begin{proof}
Let $\subtheta\subseteq \hypsurf$ be a bounded, open, and connected subset with the
property that there exists an orthonormal vector $n$ on $\subtheta$.  
Since 
 $\sigma^\top n$ is continuous on the closure $\overline{\subtheta}$, there
exists $c>0$ such that 
$\|\sigma(\xi)^\top (\xi)\|\ge c$ for all $\xi\in
\subtheta$.  

By compactness, $\hypsurf$ can be covered by finitely many sets 
$\subtheta_1,\dots,\subtheta_n$ with lower bounds $c_1,\dots,c_n$ and
we can take $c_0:=\min(c_1,\dots, c_n)$ for the conclusion to hold.
\end{proof}

Note that $\sigma(\xi)^\top n(\xi)\ne 0$ also follows from 
$\det(\sigma(\xi))\ne 0$. So in particular, regularity of $\sigma$
implies Assumption \ref{ass:sigma-ddim-sense} for compact $\hypsurf$.\\

Finally, consider Assumption \ref{ass:alpha-ddim} which asserts boundedness
of the first three derivatives of the locally defined function $\alpha$ on
$\hypsurf$.
Similar to what we have done in the proof of Proposition \ref{th:weaker-thing},
we can conclude boundedness of the derivatives from their continuity.

Assumption \ref{ass:mutilde-lipschitz}.\ref{it:mutilde-lipschitz2} is also automatically satisfied for compact $\hypsurf$.
 
From Assumption \ref{ass:C4n} one only needs the part that $\hypsurf$ is $C^4$.
Boundedness of the second and third derivative of every unit normal vector
follows from the fact that for a hypersurface a unit normal vector 
on any given connected open set is unique up to a factor of $\pm 1$ and 
from compactness.


\section{Numerical Examples}
\label{sec:Example}

In this section we present concrete examples. We compute the transform $G$ as well as the coefficients $\tilde \mu, \tilde \sigma$ of the transformed SDE to which we apply the Euler-Maruyama scheme.
Furthermore, we examine the quality of the approximation by considering the estimated $L^2$-error.

\paragraph{Discontinuity on the unit circle}

Let $\hypsurf$ be the unit circle in $\R^2$, i.e.~the drift of our SDE is discontinuous only in $\hypsurf=\{x\in \R^2:\|x\|=1\}$.
We want to solve the following SDE numerically:
\begin{align}\label{eq:kreis}
\left(
\begin{array}{c}
 dX\\ dY
\end{array}
\right)
= \mu(X,Y) dt + \sigma(X,Y) dW_t\,, \qquad
\left(
\begin{array}{c}
 X_0\\ Y_0
\end{array}
\right)
=\left(
\begin{array}{c}
 x\\ y
\end{array}
\right)\,,
\end{align}

where
\begin{align*}
 \mu(x,y)=\begin{cases}
(-x,-y)^\top, & x^2+y^2>1\\
(x,0)^\top, & x^2+y^2<1\,,
          \end{cases}
\end{align*}
$\sigma\equiv\idz$, and $W$ is a two-dimensional standard Brownian motion.
Note that the non-parallelity condition, Assumption \ref{ass:sigma-ddim-sense} is 
satisfied with $c_0=1$ ($\sigma$ is even uniformly elliptic).

We have that $p(x,y)=n(x,y)=(\sqrt{x^2+y^2})^{-1} (x,y)^\top$ yielding the transform
\begin{align*}
 G(x,y)=\begin{cases}
         \left(1+ \frac{(\sqrt{x^2+y^2}-1)|\sqrt{x^2+y^2}-1|}{\sqrt{x^2+y^2}}\phi\left(\frac{|1-\sqrt{x^2+y^2}|}{c}\right)\right)  \left(
\begin{array}{c}
 x\\y
\end{array}
\right), & (1+c)^2>x^2+y^2\ge1\\
        \left(1+ \frac{(\sqrt{x^2+y^2}-1)|\sqrt{x^2+y^2}-1|}{2\sqrt{x^2+y^2}}\phi\left(\frac{|1-\sqrt{x^2+y^2}|}{c}\right)\right)  \left(
\begin{array}{c}
 x\\y
\end{array}
\right), & (1-c)^2<x^2+y^2<1\,,\\
        \end{cases}
\end{align*}
and $G=\idz$, if $x^2+y^2 \ge (1+c)^2$, or $x^2+y^2 \le (1-c)^2$, where we have chosen $c=1/2$.\\

Then the drift of the transformed SDE is given by
\begin{align*}
\tilde \mu(G^{-1}(x,y))=
 \begin{cases}
  \nabla G(x,y) (-x,-y)^\top + \frac{1}{2} \tr(G''(x,y)), & (1+c)^2>x^2+y^2\ge1\\
  \nabla G(x,y) (x,0)^\top + \frac{1}{2} \tr(G''(x,y)), & (1-c)^2<x^2+y^2<1\,,
 \end{cases}
\end{align*}
and $\tilde \mu(x,y)=(-x,-y)^\top$, if $x^2+y^2 \ge (1+c)^2$, and
$\tilde \mu(x,y)=(x,0)^\top$, if $x^2+y^2 \le (1-c)^2$.
Furthermore, $\tilde \sigma (G^{-1}(x,y))=\nabla G(x,y)$.
$G^{-1}$ has to be evaluated numerically.\\

Figure \ref{fig:Gkreis} shows the deviation of the first component of $G$ from the identity. Figure \ref{fig:coefftildekreis} shows the first component of $\mu,\tilde \mu$, and $\sigma_{11},\tilde \sigma_{11}$.
All other components look similar.\\

\begin{figure}[ht]
\begin{center}
\includegraphics[scale=0.7]{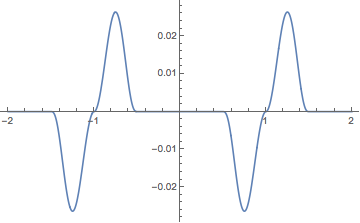}
\caption{The function $(x,y)\mapsto G_1(x,y)-x$.}\label{fig:Gkreis}
\end{center}
\end{figure}
 
\begin{figure}[ht]
\begin{center}
\includegraphics[scale=0.6]{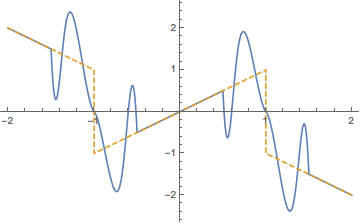}
\includegraphics[scale=0.6]{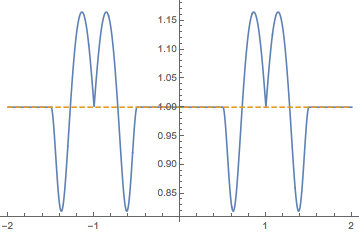}
\caption{The functions $\tilde \mu_1$ and $\tilde \sigma_{11}$ (blue line) and $\mu_1$ and $\sigma_{11}$ (yellow dashed).}\label{fig:coefftildekreis}
\end{center}
\end{figure}

We apply Algorithm \ref{alg:num-ddim} to solve SDE \eqref{eq:kreis}.
Figure \ref{fig:errorkreis} shows the estimated $L^2$-error of the approximation
of our $G$-transformed Euler-Maruyama method (GM), compared to the Euler-Maruyama (EM) scheme:
\begin{align*}
\err_k :=\log_2\left( d\, \sqrt{\hat E\left(\left\|X_T^{(k)} - X_T^{(k-1)}\right\|^2\right)}\right)
\end{align*}
plotted over $\log_2 \delta ^{(k)}$,
where $X_T^{(k)}$ is the numerical approximation with step size
$\delta=\delta^{(k)}$, $\hat E$ is an estimator of the mean value using 4096
paths, and $d$ is a normalizing constant so that $\err_1=\sqrt{1/2}$.\\

\begin{figure}[ht]
\begin{center}
\includegraphics[scale=0.6]{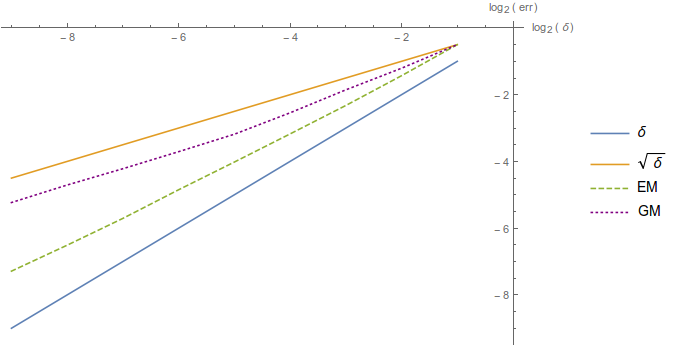}
\caption{The estimated $L^2$-error for the example where $\hypsurf$ is the unit circle.}\label{fig:errorkreis}
\end{center}
\end{figure}

We observe that our $G$-transformed (GM) method converges roughly with order $1/2$, and the crude Euler-Maruyama (EM) method seems to converge even at a higher rate.
Note however that, even though the Euler-Maruyama method is extensively used in practice, it
is not even known whether the method converges strongly for SDEs
of the kind considered here. Especially we cannot conclude whether for even
smaller step-size the
error of the Euler-Maruyama method will still become smaller, will flatten out, 
or whether it will even explode.

\paragraph{Dividend maximization}

In \cite{sz16} the dividend maximization problem from actuarial mathematics, that is, the problem of maximizing the expected discounted future dividend payments until the time of ruin $\tau$ of an insurance company, is studied. In actuarial mathematics, the solution of this optimization problem serves as a risk measure.
The problem is studied in a setup with incomplete information, where the drift of the underlying surplus process of the insurance company from which dividends are paid is driven by an unobservable Markov chain, the states of which represent different phases of the economy; an assumption that makes the model more realistic.
In order to solve the optimization problem, the underlying surplus process has to be replaced by a multidimensional process consisting of filter probabilities of the states of the hidden Markov chain and the surplus written in terms of the filter probabilities.
The resulting system is

\begin{equation}\label{eq:dividends}
\begin{aligned}
dR_t&= (\bar \alpha_t-u_t) \,dt +\beta dW_t\,,\\
d\pi_i(t)&=\left(q_{di} + \sum_{j=1}^{d-1} (q_{ji}-q_{di}) \pi_j (t)\right) \, dt + \pi_i (t) \frac{\alpha_i-\bar \alpha_t}{\beta} \, dW_t\,, \quad  i = 1, \dots, d-1 \,,
\end{aligned}
\end{equation}
where $\bar \alpha_t:=\alpha_d+\sum_{i=1}^{d-1}(\alpha_i-\alpha_d) \pi_i(t)$ and where
$(u_t)_{t\ge 0}\in [0,\bar u]$ is the dividend strategy,
$R=(R_t)_{t\ge 0}$ is the surplus process, and the $(\pi_i(t))_{t\ge 0}$, $i=1,\dots,d-1$, are the conditional probabilities that the underlying hidden Markov chain is in state $e_i$.
$W=(W_t)_{t\ge 0}$ is a one-dimensional Brownian motion.
We assume knowledge of the following constants: $(q_{ij})_{i,j=1}^d$ are the entries of the intensity matrix of the Markov chain, $\beta$ is the diffusion parameter of the surplus and $\alpha_i$, $i=1,\dots,d$, is the drift of the surplus, if the Markov chain is in state $e_i$. 

The application of filtering theory leads to an equivalent optimization problem:
\begin{align}\label{eq:opt-dividends}
\sup_u\E_{x,\pi_1,\dots,\pi_{d-1}}\left( \int_0^\tau e^{-\delta s} u_s \, ds \right)
\end{align}
with discount rate $\delta>0$.
This is studied in \cite{sz16}
and the candidate for the optimal dividend policy is of the form $u^\ast_t=\bar u \, \1_{[b(\bar \alpha_t),\infty)}(R_t)$ with \emph{threshold level} $b$, leading to a discontinuous drift of the surplus process from which the dividends are paid.
Due to the application of filtering theory, the diffusion coefficient is 
not uniformly elliptic.
In order to verify the admissibility of the candidate for the optimal control policy,
existence and uniqueness of the underlying state process has to be proven.
This can be done by applying the result presented herein and we can also simulate the optimally controlled surplus (e.g., to calculate the expected time of ruin).\\

And our results are even further applicable: in \cite{sz16} the optimization problem \eqref{eq:opt-dividends} is solved for $d=2$ by policy iteration in combination with solving an associated partial differential equation.
Doing the same for dimension 4 or higher would not be numerically tractable.
So in higher dimension one needs to solve the problem by combining policy iteration with simulation.\\

Figure \ref{fig:dividends5D-err} shows the estimated $L^2$-error of the approximation of the solution of \eqref{eq:dividends} in dimension 5 with a linear initial threshold level.
In \cite{sz16} for $d=2$ a
threshold level which is a linear interpolation of the constant optimal threshold
levels of the problem under full-information was used as an initial policy for
policy iteration. However, we need not restrict ourselves
to linear threshold levels.

Note that for our example checking whether the non-parallelity condition, Assumption \ref{ass:sigma-ddim-sense} holds (in dependence on the parameter choice) is straight-forward.

We see that in this practical example the convergence order is again roughly $1/2$.\\

\begin{figure}[ht]
\begin{center}
\includegraphics[scale=0.6]{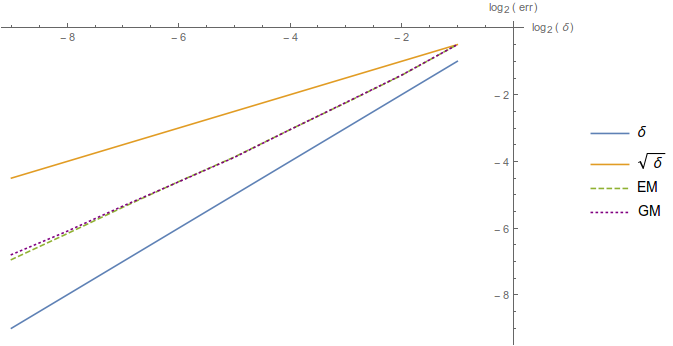}
\caption{The estimated $L^2$-error for the example of dividend maximization.}\label{fig:dividends5D-err}
\end{center}
\end{figure}

Further examples from stochastic control theory, where SDEs with discontinuous (and unbounded) drift and degenerate diffusion coefficient appear are, e.g., \cite{sz12,sz2016a,shardin2016}.
The SDEs appearing there can now be shown to have a unique global strong solution under conditions significantly weaker than known so far, and this solution can be approximated with a numerical method that converges with strong order $1/2$.
As elaborated above our method can be used for approximating solutions to these optimization problems in dimensions greater than 4, where PDE methods become practically infeasible.

\section*{Concluding remarks}

In this paper we have presented an existence and uniqueness result of strong solutions for a very general class of SDEs
with discontinuous drift und degenerate diffusion coefficient; a class of SDEs that frequently appears in applications when studying stochastic optimal control problems.
This is the most general result for such SDEs.
Furthermore, we have derived a numerical algorithm that -- under the same conditions as for the existence and uniqueness result -- is proven to converge and we have established a strong convergence order of $1/2$.
We have applied our algorithm to two examples: one of theoretical interest and one coming from a concrete optimal control problem in actuarial mathematics.


\appendix

\section{Supplementary proofs}

\subsection*{Proof of Lemma \ref{lem:nd_bounded}}

Let $\xi\in \hypsurf$. W.l.o.g. $\xi=0$ and $n(\xi)=e_d$, where $e_d$ is
the $d$-th 
canonical basis
vector of the $\R^d$. 
Thus $\hypsurf$ can locally be parametrized by $\psi:R\longrightarrow \R^d$
of the form $\psi(y_1,\ldots,y_{d-1})=(y_1,\ldots,y_{d-1},\phi(y_1,\ldots,y_{d-1}))^\top$,
where $\phi:R\longrightarrow \R$ is a $C^3$-function with $\phi(0)=0$ and
 $\phi'(0)=0$.  
Hence, for all $y\in R$, 
\begin{align*}
\lambda(y)n(\psi(y))&=\frac{\partial \psi}{\partial y_1}(y)\times\dots\times 
\frac{\partial \psi}{\partial y_{d-1}}(y)\\
&=
\begin{pmatrix}
1\\0\\\vdots\\0\\\frac{\partial \phi}{\partial y_1}(y)
\end{pmatrix}
\times
\dots\times 
\begin{pmatrix}
0\\0\\\vdots\\1\\\frac{\partial \phi}{\partial y_{d-1}}(y)
\end{pmatrix}
=
\begin{pmatrix}
-\frac{\partial \phi}{\partial y_1}(y)\\-\frac{\partial \phi}{\partial y_2}(y)\\\vdots\\-\frac{\partial \phi}{\partial y_{d-1}}(y)\\1
\end{pmatrix}\,,
\end{align*}
with $\lambda(y)=\big\|\frac{\partial \psi}{\partial y_1}(y)\times\dots\times \frac{\partial \psi}{\partial y_{d-1}}(y)\big\|$.
Note that $\lambda$ is a $C^2$-function satisfying 
$\lambda'(0)=0$ and w.l.o.g. the parametrization is chosen such that
 $\lambda(0)=1$.
Hence
\begin{align*}
\left(\lambda\ n\circ \psi\right)'(y)
&=-\begin{pmatrix}
\frac{\partial^2 \phi}{\partial y_1^2}(y)&
\frac{\partial^2 \phi}{\partial y_2\partial y_1}(y)&
\dots&
\frac{\partial^2 \phi}{\partial y_{d-1}\partial y_1}(y)\\
\frac{\partial^2 \phi}{\partial y_1\partial y_2}(y)&
\frac{\partial^2 \phi}{\partial y_2^2}(y)&
\dots&
\frac{\partial^2 \phi}{\partial y_{d-1}\partial y_2}(y)\\
\vdots&\vdots& &\vdots\\
\frac{\partial^2 \phi}{\partial y_1\partial y_{d-1}}(y)&
\frac{\partial^2 \phi}{\partial y_2\partial y_{d-1}}(y)&
\dots&
\frac{\partial^2 \phi}{\partial y_{d-1}^2}(y)\\
0&0&\dots&0
\end{pmatrix}
=
-\begin{pmatrix}
H_\phi (y)\\
0
\end{pmatrix}\,,
\end{align*}
where $H_\phi(y)$ denotes the Hessian of $\phi$ in $y\in  R$.
On the other hand 
$\left(\lambda\ n\circ \psi\right)'=n\circ \psi \ \lambda' 
+\lambda\ (n\circ \psi)'$. In particular,
\(
(n\circ\psi)'(0)=-
\left(H_\phi (0),
0\right)^\top
\)\,.\\

Now choose any $\varepsilon>0$ that is smaller than the reach 
of $\hypsurf$.
Then $\psi(0)$ is the unique closest point on $\hypsurf$ both to 
$\psi(0)+\varepsilon e_d$ and $\psi(0)-\varepsilon e_d$. In other words,
the open balls with centers $\psi(0)+\varepsilon e_d$ 
and $\psi(0)-\varepsilon e_d$ contain no point of $\hypsurf$.
Therefore, 
\[
-\varepsilon(1-\sqrt{1-(\|y\|/\varepsilon)^2})\le\phi(y)\le
\varepsilon(1-\sqrt{1-(\|y\|/\varepsilon)^2})
\]
for all $y\in R$ with $\|y\|\le \varepsilon$, from which we conclude
that\(
-\frac{\|y\|^2}{\varepsilon}\le\phi(y)\le \frac{\|y\|^2}{\varepsilon}\,,
\)
for $\|y\|$ sufficiently small. 
In particular, we have for $j\ne k$ and sufficiently small $|h|$ that
\[
-\frac{2}{\varepsilon}\le \frac{\phi(h(e_j + e_k))-\phi(h(e_j - e_k))}{2 h^2}
\le \frac{2}{\varepsilon}\,.
\] 
By letting $h\to 0$ and applying de l'Hospital's rule twice 
we see that \[
-\frac{2}{\varepsilon}\le \frac{\partial^2\phi}{\partial y_j\partial y_k}(0)
\le \frac{2}{\varepsilon}\,.
\] 
In the same way we conclude from 
\[
-\frac{1}{\varepsilon}\le \frac{\phi(he_j)-\phi(0)+\phi(-he_j)}{2 h^2}
\le \frac{1}{\varepsilon}
\] 
that 
\(
-\frac{1}{\varepsilon}\le \frac{\partial^2\phi}{\partial y_j^2}(0)
\le \frac{1}{\varepsilon}\,.
\) 
Thus 
\[
\|n'(\xi)\|^2=\| H_\phi(0)\|^2\le \sum_{j=1}^{d-1}\sum_{k=1}^{d-1}
\left|\frac{\partial^2\phi}{\partial y_j\partial y_k}(0)\right|^2
\le \sum_{j=1}^{d-1}\sum_{k=1}^{d-1} \frac{4}{\varepsilon^{2} }
=4\left(\frac{d-1}{\varepsilon}\right)^2\,,
\]
i.e.~$\|n'\|$ is bounded by $2\frac{d-1}{\varepsilon}$. 
Since this holds for all $0<\varepsilon<\reach(\hypsurf)$, we have 
$\|n'\|\le 2\frac{d-1}{\reach(\hypsurf)}$.

\subsection*{Proof of Lemma \ref{lem:itconj}}

We prove the claim that a
hypersurface that satisfies  Assumption \ref{ass:hypersurface1} 
has the property that every line segment from $x$ to $y$ can be replaced 
by a continuous curve $\gamma$ from $x$ to $y$ with 
$\ell(\gamma)<\|x-y\|+\eta$ where $\eta>0$ is a given constant.\\

Let from now on $\varepsilon<\varepsilon_0$, where $\varepsilon_0$ is as in Assumption \ref{ass:hypersurface1},
so that in particular for every $x\in \R^d$ with $d(x,\hypsurf)\le \varepsilon$ there is a unique closest point $\proj(x)$
on $\hypsurf$.

Denote by $s$ the line segment from $x$ to $y$ and identify it with 
it's parameter representation $s(t)=x+t(y-x)\|y-x\|^{-1}$.
Let $A:=\{t\in[0,\|y-x\|]:s(t)\in\hypsurf\}$. For any set $S \subseteq \R$
denote by 
$H(S)$ the set of accumulation points of $S$.

\begin{proposition}\label{prop:tangent}
Let $t\in H(A)$. Then $n(s(t))\perp s'(t)$. 
\end{proposition}

\begin{proof}
Suppose this 
was not the case, i.e.~$n(s(t))\cdot s'(t)\ne 0$. W.l.o.g. 
$n(s(t))\cdot s'(t)=C> 0$.
Let $(t_j)_{j\in \N}$ be a sequence in $A$ with $t_j\ne t$, $\lim_j t_j=t$.
W.l.o.g, $t_j>t$ for all $j$, or $t_j<t$ for all $j$.

By Assumption \ref{ass:hypersurface1} we have  
$\left(B_\varepsilon \left(s(t)- \varepsilon  n(s(t))\right)
\cup
B_\varepsilon \left(s(t)+ \varepsilon  n(s(t))\right)\right)
\cap \hypsurf=\emptyset$, where $B_r(z)$ denotes the open ball with midpoint $z$ and radius $r$. 

Suppose $t_j>t$ for all $j$. Then
\begin{align*}\|s(t)+\varepsilon n(s(t))-s(t_j)\|^2
&=\|s(t)-s(t_j)\|^2+2 \varepsilon(s(t)-s(t_j))\cdot n(s(t))+\varepsilon^2 \|n(s(t))\|^2\\
&=\|s(t)-s(t_j)\|^2+2 \varepsilon (t-t_j)s'(t)\cdot n(s(t))+\varepsilon^2\\
&=|t-t_j|^2-2 \varepsilon |t-t_j|C+\varepsilon^2 \\
&=|t-t_j|(|t-t_j|-2 \varepsilon C)+\varepsilon^2 \,,
\end{align*}
and the last expression is smaller than $\varepsilon^2$ for $j$
large enough. Thus we have found a point $\xi$ on $\hypsurf$, namely $\xi=s(t_j)$,
with $\|\xi-(s(t)+\varepsilon n(s(t)))\|<\|s(t)-(s(t)+\varepsilon n(s(t)))\|=\varepsilon $. But this contradicts the fact that 
$s(t)$ is the point on $\hypsurf$ closest to $s(t)+\varepsilon n(s(t))$.

If $t_j<t$ for all $j$, then the same argument carries through with 
$s(t)+\varepsilon n(s(t))$ replaced by
$s(t)-\varepsilon n(s(t))$.  
\end{proof}

Denote the tangent hyperplane on $\hypsurf$ in the point $\xi$
by $\thp(\xi)$, i.e.~$\thp(\xi)=\xi+\tau(\xi)=\{\xi+\vv:\vv\in \tau(\xi)\}$. 

\begin{proposition}\label{prop:segment}
For any $\xi\in \hypsurf$ we can
find $r>0$ such that for any 
$x\in \thp(\xi)$ with $\|x-\xi\|<r$ we have that
the line segment $\overline{x-\varepsilon n(\xi),x+\varepsilon n(\xi)}$ has 
precisely one 
intersection with $\hypsurf$.
\end{proposition}

\begin{proof}
We can locally parametrize $\hypsurf$ by a function on an open environment $V$
of $\xi$ in the tangent hyperplane $\thp(\xi)$. That is, there is an open 
interval $I\subseteq \R$ and a $C^2$-function 
$\hat \psi: V\longrightarrow I$ such that every point $z\in \{\xi+\vv+y n(\xi): \vv\in V, y\in I\}$ can be uniquely
written as $z=\xi+\vv+\hat \psi(\vv)n(\xi)$. 
Since $\xi\in \thp(\xi)$ and thus $\hat\psi(\xi)=0$, we may assume that
$I=(-\zeta,\zeta)$ for some $0<\zeta<\varepsilon$.
Choose some $r$ such that $0<r<\sqrt{\varepsilon^2-(\varepsilon-\zeta)^2}$
and such that for all $x\in \thp(\xi)$ we have $x\in V$ 
whenever $\|x-\xi\|<r$. 

Now if $x\in \thp(\xi)$ with $\|x-\xi\|<r$, then precisely one 
point of $\hypsurf$ lies on the line segment $\overline{x-\zeta n(\xi),x+\zeta n(\xi)}$.
But there is no point of $\hypsurf$ on the line segment
$\overline{x+\zeta n(\xi),x+\varepsilon n(\xi)}$, since this is entirely contained
in the open ball $B_\varepsilon(\xi+\varepsilon n(\xi))$, which by the
unique closest point property for $\xi+\varepsilon n(\xi)$ does not
contain any point of $\hypsurf$.

By the same reasoning $\overline{x-\zeta n(\xi),x-\varepsilon n(\xi)}\cap \hypsurf=\emptyset$.
\end{proof}

\begin{proposition}\label{prop:app3}
Let $\varepsilon_1 < \varepsilon$.
Then for any $y\in \R^d$ 
there exists a point $\hat y\in \R^d$ with $d(\hat y,\hypsurf)\ge \varepsilon_1$ and
$\|y-\hat y\|\le \varepsilon_1$. 
\end{proposition}

\begin{proof}
If $d(y,\hypsurf)\ge \varepsilon_1$, then set $\hat y=y$.
Otherwise, there is a unique closest point $p(y)\in \hypsurf$.
Set 
\[
\hat y=\begin{cases}
p(y)+\varepsilon_1 n(p(y)) & \text{if } n(p(y))\cdot (y-p(y))>0\\ 
p(y)-\varepsilon_1 n(p(y)) & \text{if } n(p(y))\cdot (y-p(y))<0\,. 
\end{cases}
\]
Then $\|y-\hat y \| \le \varepsilon_1$ is obvious, and $d(\hat y,\hypsurf) \ge \varepsilon_1$ by the unique closest point property.
\end{proof}

We can now modify the straight line from $x$ to $y$ to get a 
continuous curve, which is not much longer than $\|y-x\|$, but
has only finitely many intersections with $\hypsurf$.\\

For what follows, let $\alpha\in (0,1)$ and for $0<\delta<\varepsilon$ set
$\varepsilon_1=\varepsilon-\sqrt{\varepsilon^2-\delta^2}$. \\

We construct a sequence $(\gamma_k)_{k\in \N_0}$
of continuous curves of finite length which becomes stationary 
after finitely many steps, i.e.~there exists 
$k_0$ such that $\gamma_k=\gamma_{k_0}$ for all $k\ge k_0$.

Furthermore, $\gamma_{k_0}$ will have
only finitely many intersections with $\hypsurf$ and it will be only slightly longer than $\|x-y\|$, see \eqref{eq:length}.

Set $\gamma_0=s$.

\paragraph{Step $1$:} 
{\em If} $H(s\cap \hypsurf)=\emptyset$, then set $\gamma_1=\gamma_0$.

{\em Otherwise} proceed as follows:
According to Proposition \ref{prop:app3} there exists a 
point $\hat y$ with $d(\hat y,\hypsurf)\ge \varepsilon_1$ and
$\|y-\hat y\|\le \varepsilon_1$. 
Define $\gamma_1$ as the concatenation of the lines 
$\overline{x,\hat y}$ and $\overline{\hat y,y}$. We have 
$\ell(\gamma_1)\le \|y-x\|+2\varepsilon_1$, and there is at most one
intersection of $\overline{\hat y, y}$, the second line segment, with $\hypsurf$, due to Assumption \ref{ass:hypersurface1}.
Set $x_1=x$.\\

After step 1 we have constructed a polygonal curve $\gamma_1$ such that 
 $\ell(\gamma_1)\le \|y-x\|+2 \varepsilon_1$. If $\gamma_1$ has infinitely
many intersections with $\hypsurf$, then all but finitely many are
contained in a single 
line segment, $s_1=\overline{x_1,\hat y}$, which satisfies
$\ell(s_1)=\|\hat y-x_1\|=\|\hat y-x\|=\|(y-x)+(\hat y-y)\|\le \|y-x\|+ \varepsilon_1$.\\

Now we enter an iteration procedure.
Suppose that after $k\ge 1$ steps we have constructed 
a polygonal curve $\gamma_k$, with the properties that 
 $\ell(\gamma_k)\le \|y-x\|+2 k \varepsilon_1$, and such that 
either $\gamma_k$ has finitely
many intersections with $\hypsurf$, or all intersections are 
contained in a single 
line segment, $s_k=\overline{x_k,\hat y}$, which satisfies
$\ell(s_k)\le \|y-x\|- (k-1) (\alpha\delta-\varepsilon_1)+ \varepsilon_1$.\\

We construct $\gamma_{k+1}$ from $\gamma_k$ as follows:

\paragraph{Step $k+1$:} 
{\em If} $H(\gamma_k\cap \hypsurf)=\emptyset$, then set $\gamma_{k+1}=\gamma_k$.
 
{\em Otherwise}, $H(\gamma_k\cap \hypsurf)$ is contained in the line segment
$\overline {x_k,\hat y}$. Parametrize this segment by 
$s_k(t)=x_k+t\|\hat y-x_k\|^{-1}(\hat y- x_k)$, $t\in [0,\|\hat y-x_k\|]$ 
and let $H_k=H\left(\{t: s_k(t)\in \hypsurf\}\right)$.

Set $t_k=\min H_k$, and let $n_k=n(s_k(t_k))$. 
If $t_k$ is isolated from the left,
or if $t_k=0$, then set $r_k=0$. Now consider the case where $t_k$ is not isolated from the left.
By Proposition \ref{prop:tangent}, $s_k$ lies in the tangent hyperplane
$\thp(s_k(t_k))=s_k(t_k)+\tau(s_k(t_k))$ and 
we can find a small ball with radius $r_k>0$ such that, for any
$t$ with $|t-t_k|<r_k$, the line
segment $\overline{s_k(t)-\varepsilon_1 n_k,s_k(t)+\varepsilon_1 n_k}$ has
at most one intersection with $\hypsurf$, by Proposition \ref{prop:segment}. \\

Consider the line segment 
$\overline{s_k(t_k-r_k)+\varepsilon_1
n_k,s_k(t_k+\alpha\delta)+\varepsilon_1 n_k}$. 

{\em If} the intersection
of this with the plane through $\hat y$, which is orthogonal to the line
segment, is non-empty, denote the unique intersection point by 
$z_k$.

{\em Then} we construct $\gamma_{k+1}$ as the concatenation of
the following line segments:
\begin{itemize}
\item $\overline {s_k(0),s_k(t_k-r_k)}$, which by definition of $t_k$ and $r_k$
has only finitely many intersections with $\hypsurf$;
\item $\overline {s_k(t_k-r_k),s_k(t_k-r_k)+\varepsilon_1 n_k}$, which has
at most one intersection with $\hypsurf$ by the construction of $r_k$ and Proposition \ref{prop:segment};
\item $\overline {s_k(t_k-r_k)+\varepsilon_1
n_k,z_k}$, which is completely
contained in $B_\varepsilon (s_k(t_k)+\varepsilon n_k)$, which does not contain any point of $\hypsurf$ by the unique
closest point property for $s_k(t_k)+\varepsilon n_k$;
\item $\overline{z_k,\hat y}$, which has no intersection with $\hypsurf$, because
as $\|z_k-\hat y\|= \varepsilon_1$, there is no intersection strictly
between $z_k$ and $\hat y$, and $z_k$ lies in the closure of $B_{\varepsilon_1}(\hat y)$ (this is where we need Step 1);
\item $\overline{\hat y,y}$.
\end{itemize}

In this case the curve $\gamma_{k+1}$ has only finitely many intersections with
$\hypsurf$ and $\ell(\gamma_{k+1})= \ell(\gamma_k)+2 \varepsilon_1\le \|y-x\|+2(k+1)
\varepsilon_1$.\\

{\em Otherwise}, set $x_{k+1}=s_k(t_k+\alpha\delta)+\varepsilon_1 n_k$, and construct $\gamma_{k+1}$ as the concatenation of
the following line segments:
\begin{itemize}
\item $\overline {s_k(0),s_k(t_k-r_k)}$, which by definition of $t_k$ and $r_k$
has only finitely intersections with $\hypsurf$;
\item $\overline {s_k(t_k-r_k),s_k(t_k-r_k)+\varepsilon_1 n_k}$, which has
at most one intersection with $\hypsurf$ 
by the construction of $r_k$ and Proposition \ref{prop:segment};
\item $\overline {s_k(t_k-r_k)+\varepsilon_1
n_k,x_{k+1}}$, which is completely
contained in $B_\varepsilon (s_k(t_k)+\varepsilon n_k)$, which does not contain any point of $\hypsurf$ by the unique
closest point property for $s_k(t_k)+\varepsilon n_k$;
\item $s_{k+1}:=\overline {x_{k+1},\hat y}$, which still may have infinitely many intersections with $\hypsurf$;
\item $\overline{\hat y,y}$.
\end{itemize}
Again we have that $\ell(\gamma_{k+1})\le \ell(\gamma_k)+2 \varepsilon_1\le 
\|y-x\|+2 (k+1) \varepsilon_1$. Note that 
\begin{align*}
\ell(s_{k+1})^2&=\|x_{k+1}-\hat y\|^2
=\|s_k(t_k+\alpha\delta)+\varepsilon_1 n_k-\hat y\|^2\\
&= \|s_k(t_k+\alpha\delta)-\hat y\|^2+\varepsilon_1^2
=(\|s_k(t_k)-\hat y\|-\alpha\delta)^2+ \varepsilon_1^2 \,.
\end{align*}
In particular, 
$\|x_{k+1}-\hat y\|
\le \big|\|s_k(t_k)-\hat y\|-\alpha\delta\big|+\varepsilon_1
= \|s_k(t_k)-\hat y\|-\alpha\delta+\varepsilon_1$.
Note that $\|s_k(t_k)-\hat y\|-\alpha\delta\ge 0$, 
since otherwise the line segment 
$\overline{s_k(t_k-r_k)+\varepsilon_1
n_k,s_k(t_k+\alpha\delta)+\varepsilon_1 n_k}$ 
would intersect the hyperplane orthogonal to $s_k$ and passing 
through $\hat y$.

Thus $\|x_{k+1}-\hat y\|\le \|x_k-\hat y\|-\alpha \delta +\varepsilon_1
\le \|x- y\|-k(\alpha\delta-\varepsilon_1)+\varepsilon_1$.\\

After step $k+1$ we have constructed a polygonal curve 
$\gamma_{k+1}$ such that 
 $\ell(\gamma_{k+1})\le \|y-x\|+2(k+1)\varepsilon_1$. If $\gamma_{k+1}$ 
has infinitely
many intersections with $\hypsurf$, then all but finitely many are 
contained in a single 
line segment, $s_{k+1}=\overline{x_{k+1},\hat y}$, and $\ell(s_{k+1})\le 
\|x- y\|-k(\alpha \delta-\varepsilon_1) +\varepsilon_1$.\\

So finally we have constructed a sequence $(\gamma_k)_{k\in \N_0}$ with 
\vspace{-.5em}
\begin{itemize}\setlength\itemsep{0em}
\item[-] 
$\ell(\gamma_k)\le \|x-y\|+2 k\varepsilon_1$; 
\item[-] $\gamma_k$ either has only finitely many intersections with $\hypsurf$,
or all but finitely many intersections are contained in a segment of 
length at most 
$\|x-y\|-(k-1)(\alpha \delta -\varepsilon_1)+\varepsilon_1$.
\end{itemize}

Since $\delta<\varepsilon$, we have that
$\varepsilon_1=\varepsilon-\sqrt{\varepsilon^2-\delta^2}
=\varepsilon\left(1-\sqrt{1-\left(\frac{\delta}{\varepsilon}\right)^2}\right)
<\frac{\delta^2}{\varepsilon}$, such that 
\begin{equation*}
\alpha\delta-\varepsilon_1
>\delta\left(\alpha-\frac{\delta}{\varepsilon}\right)> 0\,.
\end{equation*}

With this, and since 
$\|x-y\|-(k-1)(\alpha \delta -\varepsilon_1)+\varepsilon_1\ge \ell(s_k)\ge 0$,
the iteration can have at most
\[
k\le1+\frac{\|x-y\|+\varepsilon_1}{2(\alpha \delta -\varepsilon_1)}
<1+\frac{\|x- y\|+\varepsilon_1}{2\delta\left(\alpha-\frac{\delta}{\varepsilon}\right)}
<1+\frac{\|x- y\|+\varepsilon}{2\delta\left(\alpha-\frac{\delta}{\varepsilon}\right)}
\]
steps before the sequence becomes stationary, and 
thus there exists a $k_0$ such that $\gamma_{k_0}$ has at most
finitely many intersections with $\hypsurf$.

For the length of $\gamma_k$ for $k\ge k_0$ we have
\begin{align}\label{eq:length}
\ell(\gamma_k)&\le \|x-y\|+2k\varepsilon_1 
\le \|x-y\|
+\left(2\delta+\frac{\|x-y\|+\varepsilon}{\alpha  -\frac{\delta}{\varepsilon}}\right)
\frac{\delta}{\varepsilon}\,.
\end{align}
This can be made as close to $\|x-y\|$ as we desire by making $\delta$ small. Thus the proof is finished.


\section*{Acknowledgements}

The authors thank Bert J\"uttler for valuable discussions regarding 
questions related to differential geometry and Thomas M\"uller-Gronbach for pointing out an inaccuracy in the definition of $\alpha$.

G. Leobacher is supported by the Austrian Science Fund (FWF): Project F5508-N26, which is part of the Special Research Program "Quasi-Monte Carlo Methods: Theory and Applications".
This paper was written while G. Leobacher was member of the Department of Financial Mathematics and Applied Number Theory, Johannes Kepler University Linz, 4040 Linz, Austria.

M. Sz\"olgyenyi is supported by the Vienna Science and Technology Fund (WWTF): Project MA14-031.
A part of this paper was written while M. Sz\"olgyenyi was member of the Department of Financial Mathematics and Applied Number Theory, Johannes Kepler University Linz, 4040 Linz, Austria.
During this time, M. Sz\"olgyenyi was supported by the Austrian Science Fund (FWF): Project F5508-N26, which is part of the Special Research Program "Quasi-Monte Carlo Methods: Theory and Applications".



\vspace{2em}
\centerline{\underline{\hspace*{16cm}}}

\noindent G. Leobacher \\
Institute of Mathematics and Scientific Computing, University of Graz, Heinrichstra\ss{}e 36, 8010 Graz, Austria\\
gunther.leobacher@uni-graz.at\\

\noindent M. Sz\"olgyenyi \Letter \\
Institute of Statistics and Mathematics, Vienna University of Economics and Business, Welthandelsplatz 1, 1020 Vienna, Austria\\
michaela.szoelgyenyi@wu.ac.at


\end{document}